\documentclass[a4paper,12pt,reqno]{amsart}
\usepackage{amsmath, amssymb,amsthm}
\usepackage{enumerate}
\usepackage{cite}
\usepackage{graphicx}
\usepackage{subcaption}
\nonstopmode \numberwithin{equation}{section}
\theoremstyle{plain}
\newtheorem{theorem}{Theorem}[section]

\newtheorem{corollary}{Corollary}[section]

\newtheorem{remark}{Remark}[section]
\newtheorem{proposition}{Proposition}[section]
\theoremstyle{definition}
\newtheorem{example}{Example}[section]
\newtheorem{note}{Note}
\begin{document}
\bibliographystyle{amsplain}
\title{{{
Orthogonal Polynomials related to $\MakeLowercase{g}$-fractions with missing terms
}}}
\author{
Kiran Kumar Behera
}
\address{Department of Mathematics\\
 Indian Institute of Technology Roorkee-247 667,
 Uttarakhand, India}
\email{krn.behera@gmail.com}

\author{
A. Swaminathan
}
\address{
Department of  Mathematics  \\
Indian Institute of Technology Roorkee-247 667,
Uttarakhand, India
}
\email{swamifma@iitr.ac.in, mathswami@gmail.com}

\bigskip

\begin{abstract}
The purpose of the present paper is to investigate some
structural and qualitative aspects of two different
perturbations of the parameters of $g$-fractions.
In this context the concept of \emph{gap} $g$-fractions is introduced. While tail sequences of a continued fraction play a
significant role in the first perturbation,
Schur fractions are used
in the second perturbation of the $g$-parameters that are considered.
Illustrations are provided using Gaussian hypergeometric functions.
Using a particular gap $g$-fraction, some members
of the class of Pick functions are also identified.
\end{abstract}



\maketitle

\pagestyle{myheadings} \markboth{ Kiran Kumar Behera and A. Swaminathan }
{Orthogonal Polynomials related to $\MakeLowercase{g}$-fractions with missing terms}
\section{Introduction}
\label{sec:Introduction}
Given an arbitrary real sequence $\{g_k\}_{k=0}^{\infty}$,
a continued fraction expansion of the form
\begin{align}
\label{eqn:definition of g-fraction expansion}
\dfrac{1}{1}
\begin{array}{cc}\\$-$\end{array}
\dfrac{(1-g_0)g_1z}{1}
\begin{array}{cc}\\$-$\end{array}
\dfrac{(1-g_1)g_2z}{1}
\begin{array}{cc}\\$-$\end{array}
\dfrac{(1-g_2)g_3z}{1}
\begin{array}{cc}\\$-$\end{array}
\cdots,
\quad z\in\mathbb{C},
\end{align}
is called a $g$-fraction if the parameters
$g_j\in[0,1]$, $j\in\mathbb{N}\cup\{0\}$.
It terminates and equals a rational function if
$g_j\in\{0,1\}$ for some $j\in\mathbb{N}\cup\{0\}$.
If $0<g_j<1$, $j\in\mathbb{N}\cup\{0\}$,
the $g$-fraction \eqref{eqn:definition of g-fraction expansion}
still converges uniformly on every compact subsets of the slit
domain $\mathbb{C}\setminus[1,\infty)$
(see \cite[Theorem 27.5]{Wall_book} and
\cite[Corollary 4.60]{Jones-Thron-book}),
and in this case,
\eqref{eqn:definition of g-fraction expansion}
will represent an
analytic function, say $\mathcal{F}(z)$.

Such $g$-fractions are found having applications in diverse areas
like number theory
\cite{Alexei-Runckel-points-RJ},
dynamical systems
\cite{Alexei-ABC-flow-2013-JAT},
moment problems and analytic
function theory
\cite{Swami-mapping-prop-BHF-2014-JCA,
Kustner-g-fractions-2002-CMFT,
Kustner-g-fractions-JMAA-2007}.
In particular, \cite[Theorem 69.2]{Wall_book},
the Hausdorff moment problem
\begin{align*}
\nu_j=\int_{0}^{1}\sigma^jd\nu(\sigma),
\quad j\geq0,
\end{align*}
has a solution if and only if \eqref{eqn:definition of g-fraction expansion}
corresponds to a power series of the form
$1+\nu_1z+\nu_2z^2+\cdots$, $z\in\mathbb{C}\setminus[1,\infty)$.
Further, the $g$-fractions have also been used to study the geometric properties of ratios of Gaussian hypergeometric functions as well as their $q$-analogues,
(see the proofs of \cite[Theorem 1.5]{Kustner-g-fractions-2002-CMFT}
and \cite[Theorem 2.2]{Swami-mapping-prop-BHF-2014-JCA}).
Among several such results,
one of the most fundamental result concerning $g$-fraction is
\cite[Theorem 74.1]{Wall_book} in which holomorphic functions
having positive real part in $\mathbb{C}\setminus[1, \infty)$ are characterised.
Precisely, $\mathrm{Re}(\sqrt{1+z}\,\,\mathcal{F}(z))$ is positive
if and only if $\mathcal{F}(z)$ has a continued fraction expansion
of the form \eqref{eqn:definition of g-fraction expansion}.
Moreover, $\mathcal{F}(z)$ has the integral representation
\begin{align*}
\mathcal{F}(z)=\int_{0}^{1}\dfrac{d\phi(t)}{1-zt},
\quad
z\in\mathbb{C}\setminus[1,\infty],
\end{align*}
where $\phi(t)$ is a bounded non-decreasing function
having a total increase 1.

Many interesting results are also available in literature if we consider
subsets of $\mathbb{C}\setminus[1,\infty)$.
For instance, let $\mathbb{K}$ be the class of
holomorphic functions having a positive real part on the unit disk
$\mathbb{D}:=\{z:|z|<1\}$. Such functions denoted by $\mathcal{C}(z)$
are called Carath\'{e}odory functions and have the Riesz-Herglotz
representation \cite[Theorem 73.1]{Wall_book}
\begin{align*}
\mathcal{C}(z)=
\int_{0}^{2\pi}\dfrac{e^{it}+z}{e^{it}-z}d\phi(t)+qi,
\end{align*}
where $q$=$\mathrm{Im}\,\,\mathcal{C}(0)$.
Further, if $\mathcal{C}(z)\in\mathbb{K}$ is such that
$\mathcal{C}(\mathbb{R})\subseteq \mathbb{R}$
and normalised by $\mathcal{C}(0)=1$,
then the following continued fraction expansion can be derived
\cite{Wall-cf-and-bdd-analytic-function-1944-BAMS}
\begin{align}
\label{eqn:g-fraction expansion for carat function}
\dfrac{1-z}{1+z}C(z)=
\dfrac{1}{1}
\begin{array}{cc}\\$-$\end{array}
\dfrac{g_1\omega}{1}
\begin{array}{cc}\\$-$\end{array}
\dfrac{(1-g_1)g_2\omega}{1}
\begin{array}{cc}\\$-$\end{array}
\dfrac{(1-g_2)g_3\omega}{1}
\begin{array}{cc}\\$-$\end{array}
\cdots,
\quad z\in\mathbb{D},
\end{align}
where $w=-4z/(1-z)^2$. Note that here $g_0=0$.

Closely related to the Carath\'{e}odory functions are the Schur functions
$f(z)$ given by
\begin{align}
\label{eqn:relation between Schur and Carat functions}
C(z)=\dfrac{1+zf(z)}{1-zf(z)},
\quad z\in\mathbb{D}.
\end{align}
From \eqref{eqn:relation between Schur and Carat functions},
it is clear that $f(z)$ maps the unit disk $\mathbb{D}$
to the closed unit disk $\bar{\mathbb{D}}$.
In fact, if
\begin{align*}
\mathbb{B}=\{\mbox{f: f is holomorphic and }
             f(\mathbb{D})\subseteqq\bar{\mathbb{D}}\},
\end{align*}
\eqref{eqn:relation between Schur and Carat functions}
describes a one-one correspondence between the classes of
holomorphic functions $\mathbb{B}$ and $\mathbb{K}$.
The class $\mathbb{B}$ studied by J. Schur
\cite{Schur-papers-1917-1918},
is the well-known Schur algorithm.
This algorithm generates a sequence of rational functions
$\{f_n(z)\}_{n=0}^{\infty}$ from a given sequence
$\{\alpha_n\}_{n=0}^{\infty}$ of complex numbers lying in
$\bar{\mathbb{D}}$.
Then, with $\alpha_n$, $n\geq0$, satisfying some
positivity conditions, $f_n(z)\rightarrow f(z)$,
$n\rightarrow\infty$, where $f(z)\in\mathbb{B}$
is a Schur function
\cite{JNT-Survey-Schur-PC-Szego}.

It is interesting to note that $\alpha_n=f_n(0)$,
$n\geq0$ where $f_0(z)\equiv f(z)$. Moreover, using these
parameters $\{\alpha_n\}_{n\geq0}$,
the following Schur fraction can be
obtained \cite{JNT-Survey-Schur-PC-Szego}
\begin{align}
\label{eqn:definition of Schur fraction}
\alpha_{0}+
\dfrac{(1-|\alpha_0|^2)z}{\bar{\alpha}_0z}
\begin{array}{cc}\\$+$\end{array}
\dfrac{1}{\alpha_1}
\begin{array}{cc}\\$+$\end{array}
\dfrac{(1-|\alpha_1|^2)z}{\bar{\alpha}_1z}
\begin{array}{cc}\\$+$\end{array}
\dfrac{1}{\alpha_2}
\begin{array}{cc}\\$+$\end{array}
\dfrac{(1-|\alpha_2|^2)z}{\bar{\alpha}_2z}
\begin{array}{cc}\\$+$\end{array}
\cdots,
\end{align}
where $\alpha_j$ is related to the $g_j$
occurring in \eqref{eqn:g-fraction expansion for carat function}
by $\alpha_j=1-2g_j$, $j\geq1$.
Similar to $g$-fractions the Schur
fraction also terminates if $|\alpha_n|=1$ for some $n\in\mathbb{Z}_{+}$.
It may be noted that such a case occurs if and only if
$f(z)$ is a finite Blashke product \cite{JNT-Survey-Schur-PC-Szego}.

Let $\mathrm{A}_n(z)$ and $\mathrm{B}_n(z)$ denote the $n^{th}$ partial
numerator and denominator of
\eqref{eqn:definition of Schur fraction} respectively.
Then, with the initial values $\mathrm{A}_0(z)=\alpha_0$,
$\mathrm{B}_0(z)=1$, $\mathrm{A}_1(z)=z$
and
$\mathrm{B}_1(z)=\bar{\alpha}_0z$,
the following recurrence relations hold
\cite{JNT-Survey-Schur-PC-Szego}
\begin{align}
\label{eqn:Wallis formular for even approximants-Schur fraction}
\mathrm{A}_{2n}(z)&=\alpha_n\mathrm{A}_{2n-1}(z)+\mathrm{A}_{2n-2}(z)\nonumber\\
\mathrm{B}_{2n}(z)&=\alpha_n\mathrm{B}_{2n-1}(z)+\mathrm{B}_{2n-2}(z),
\quad n\geq1,
\end{align}
\begin{align}
\label{eqn:Wallis formular for odd approximants-Schur fraction}
\mathrm{A}_{2n+1}(z)&=\bar{\alpha}_nz\mathrm{A}_{2n}(z)+(1-|\alpha_n|^2)z\mathrm{A}_{2n-1}(z)\nonumber\\
\mathrm{B}_{2n+1}(z)&=\bar{\alpha}_nz\mathrm{B}_{2n}(z)+(1-|\alpha_n|^2)z\mathrm{B}_{2n-1}(z),
\quad n\geq1.
\end{align}
Using
\eqref{eqn:Wallis formular for even approximants-Schur fraction}
in
\eqref{eqn:Wallis formular for odd approximants-Schur fraction},
we get
\begin{align}
\label{eqn:Wallis formula for 2n+1, 2n-1, 2n-2}
\mathrm{A}_{2n+1}(z)&=z\mathrm{A}_{2n-1}(z)+\bar{\alpha}_pz\mathrm{A}_{2n-2}(z)\nonumber\\
\mathrm{B}_{2n+1}(z)&=z\mathrm{B}_{2n-1}(z)+\bar{\alpha}_pz\mathrm{B}_{2n-2}(z).
\end{align}
The relations
\eqref{eqn:Wallis formular for even approximants-Schur fraction}
and
\eqref{eqn:Wallis formula for 2n+1, 2n-1, 2n-2}
are sometimes written in the more
precise matrix form as
\begin{align}
\label{eqn:matrix relation for Schur approximants}
\left(
  \begin{array}{cc}
    \mathrm{A}_{2p+1} & \mathrm{B}_{2p+1} \\
    \mathrm{A}_{2p} & \mathrm{B}_{2p} \\
  \end{array}
\right)=
\left(
  \begin{array}{cc}
    z & \bar{\alpha}_pz \\
    \alpha_p & 1 \\
  \end{array}
\right)
\left(
  \begin{array}{cc}
    \mathrm{A}_{2p-1} & \mathrm{A}_{2p-1} \\
    \mathrm{A}_{2p-2} & \mathrm{A}_{2p-2} \\
  \end{array}
\right),
\quad p\geq1.
\end{align}

It is also known that
\cite{Njaastad-convernce-of-schur-algo-1990-PAMS}
\begin{align}
\label{eqn:reciprocal relations between An and Bn}
\mathrm{A}_{2n+1}(z)&=z\mathrm{B}_{2n}^{\ast}(z)
\quad ; \quad
\mathrm{B}_{2n+1}(z)=z\mathrm{A}_{2n}^{\ast}(z)\nonumber\\
\mathrm{A}_{2n}(z)&=\mathrm{B}_{2n+1}^{\ast}(z)
\quad ; \quad
\mathrm{B}_{2n}(z)=\mathrm{A}_{2n+1}^{\ast}(z).
\end{align}
Here, and in what follows,
$\mathrm{P}_n^{\ast}(z)=z^n\overline{\mathrm{P}_n(1/\bar{z})}$
for any polynomial $\mathrm{P}_n(z)$
with complex coefficients and of degree $n$.
From \eqref{eqn:reciprocal relations between An and Bn},
it follows that
\begin{align*}
\dfrac{\mathrm{A}_{2n+1}(z)}{\mathrm{B}_{2n+1}(z)}=
\left(\dfrac{\mathrm{A}_{2n}^{\ast}(z)}{\mathrm{B}_{2n}^{\ast}(z)}\right)^{-1},
\quad
\dfrac{\mathrm{A}_{2n}(z)}{\mathrm{B}_{2n}(z)}=
\left(\dfrac{\mathrm{A}_{2n+1}^{\ast}(z)}{\mathrm{B}_{2n+1}^{\ast}(z)}\right)^{-1}.
\end{align*}
Further, the even approximants of the Schur fraction
\eqref{eqn:definition of Schur fraction}
coincide with the the $n^{th}$ approximant of the Schur algorithm,
so that
$\mathrm{A}_{2n}(z)/\mathrm{B}_{2n}(z)$ converges to the
Schur function $f(z)$ as $n\rightarrow\infty$.

From the point of view of their applications, it is obvious
that the parameters $g_n$ of the $g$-fraction
\eqref{eqn:definition of g-fraction expansion} contain hidden
information about the properties of the dynamical systems or the
special functions they represent.
One way to explore this hidden
information is through perturbation;
that is, through a study of the consequences when some disturbance
is introduced in the parameter sequence $\{g_n\}$.
The main objective of the present manuscript is to study the
structural and qualitative aspects of two
perturbations.
The first is when a finite number of parameters $g_j$'s are missing
in which case we call the corresponding $g$-fraction as a \emph{gap}-$g$-fraction.
The second case is replacing $\{g_n\}_{n=0}^{\infty}$ by
a new sequence $\{g_n^{(\beta_k)}\}_{n=0}^{\infty}$ in which the
$j^{th}$ term $g_j$ is replaced by $g_j^{(\beta_k)}$.
The first case is illustrated
using Gaussian hypergeometric functions,
where we use the fact that
many $g$-fractions converge to ratios of
Gaussian hypergeometric functions in slit complex domains.
The second case is studied by applying the technique of coefficient
stripping \cite{Simon-book-vol1} to the sequence of Schur parameters $\{\alpha_j\}$.
This follows from the fact that the Schur fraction and the $g$-fraction are
completely determined by the related Schur parameters $\alpha_k$'s
and the $g$-parameters respectively, and that a perturbation in $\alpha_j$
produces a unique change in the $g_j$ and vice-versa.

The manuscript is organised as follows.
Section~\ref{sec:gap g fractions and structural relations}
provides structural relations for the three different cases
provided by the gap $g$-fraction. A particular ratio
of Gaussian hypergeometric functions is used to illustrate the results.
The modified $g$-fractions given in
Section~\ref{sec:gap g fractions and structural relations}
has a shift in $g_k$ to $g_{k+1}$ and so on for any fixed $k$.
Instead, the effect of changing $g_k$ to any another value
$g_k^{(\beta_k)}$ is discussed in
Section~\ref{sec:Perturbed Schur parameters}.
Illustrations of the results obtained in
Section~\ref{sec:gap g fractions and structural relations}
leading to characterization of a class of ratio of hypergeometric functions
such as Pick functions is outlined in
Section~\ref{sec:a class of pick functions and schur functions}.
\section{Gap $g$-fractions and structural relations}
\label{sec:gap g fractions and structural relations}
As the name suggests, gap-$g$-fractions correspond to the
$g$-sequence $\{g_k\}_{k=0}^{\infty}$
with missing parameters.
We study three cases in this section and in each the concept of tail sequences
of a continued fraction plays an important role. For more information on the tails
of a continued fraction, we refer to
\cite[Chapter II]{Lisa-Waadeland-book-cf-with-application}.

For $z\in\mathbb{C}\setminus[1,\infty)$, let $\mathcal{F}(z)$ be the continued fraction
\eqref{eqn:definition of g-fraction expansion} and
\begin{align}
\label{eqn:g-fraction expansion with gk missing}
\lefteqn{\mathcal{F}(k;z)=}\nonumber\\
&&
\dfrac{1}{1}
\begin{array}{cc}\\$-$\end{array}
\dfrac{(1-g_0)g_1z}{1}
\begin{array}{cc}\\$-$\end{array}
\cdots
\dfrac{(1-g_{k-2})g_{k-1}z}{1}
\begin{array}{cc}\\$-$\end{array}
\dfrac{(1-g_{k-1})g_{k+1}z}{1}
\begin{array}{cc}\\$-$\end{array}
\dfrac{(1-g_{k+1})g_{k+2}z}{1}
\begin{array}{cc}\\$-$\end{array}
\cdots.
\end{align}
Note that
\eqref{eqn:g-fraction expansion with gk missing}
is obtained from
\eqref{eqn:definition of g-fraction expansion}
by removing $g_k$ for some arbitrary $k$
which cannot be obtained by letting $g_k=0$.
Let,
\begin{align}
\label{eqn:tail sequence for g fraction with gk missing}
\mathcal{H}_{k+1}(z)=
\dfrac{g_{k+1}z}{1}
\begin{array}{cc}\\$-$\end{array}
\dfrac{(1-g_{k+1})g_{k+2}z}{1}
\begin{array}{cc}\\$-$\end{array}
\dfrac{(1-g_{k+2})g_{k+3}z}{1}
\begin{array}{cc}\\$-$\end{array}
\cdots,
\end{align}
so that $-(1-g_k)\mathcal{H}_{k+1}(z)$ is the $(k+1)^{th}$
tail of $\mathcal{F}(z)$.
We note that
\cite[Theorem 1, p.56]{Lisa-Waadeland-book-cf-with-application},
the existence of $\mathcal{F}(z)$ guarantees the
existence of $\mathcal{H}_{k+1}(z)$.
Further, if
\begin{align*}
h(k;z)=(1-g_{k-1})\mathcal{H}_{k+1}(z),
\quad k\geq1,
\end{align*}
then, from \eqref{eqn:g-fraction expansion with gk missing}
and \eqref{eqn:tail sequence for g fraction with gk missing}
we obtain the rational function
\begin{align}
\label{eqn:tail sequence for g fraction with gk missing and terminating with hkz}
\dfrac{\mathcal{X}_{k}(h(k;z);z)}{\mathcal{Y}_{k}(h(k;z);z)}=
\dfrac{1}{1}
\begin{array}{cc}\\$-$\end{array}
\dfrac{(1-g_0)g_1z}{1}
\begin{array}{cc}\\$-$\end{array}
\cdots
\dfrac{(1-g_{k-2})g_{k-1}z}{1-h(k;z)}.
\end{align}
It is known
\cite[Theorem 2.1]{Jones-Thron-book},
that the $k^{th}$ approximant of
\eqref{eqn:definition of g-fraction expansion}
is given by the rational function
\begin{align*}
\dfrac{\mathcal{X}_{k}(0;z)}{\mathcal{Y}_{k}(0;z)}=
\dfrac{1}{1}
\begin{array}{cc}\\$-$\end{array}
\dfrac{(1-g_0)g_1z}{1}
\begin{array}{cc}\\$-$\end{array}
\cdots
\dfrac{(1-g_{k-2})g_{k-1}z}{1}=
\mathcal{S}_{k}(0,z),
\quad\mbox{say},
\end{align*}
and that
\begin{align*}
\dfrac{\mathcal{X}_{k}(h(k;z);z)}{\mathcal{Y}_{k}(h(k;z);z)}=
\dfrac{\mathcal{X}_{k}(0;z)-h(k;z)\mathcal{X}_{k-1}(0;z)}
      {\mathcal{Y}_{k}(0;z)-h(k;z)\mathcal{Y}_{k-1}(0;z)}.
\end{align*}
Then,
\begin{align*}
\dfrac{\mathcal{X}_{k}(h(k;z);z)}{\mathcal{Y}_{k}(h(k;z);z)}-
\dfrac{\mathcal{X}_{k}(0;z)}{\mathcal{Y}_{k}(0;z)}
&=
\dfrac{h(k;z)[\mathcal{X}_{k}(0;z)\mathcal{Y}_{k-1}(0;z)-\mathcal{X}_{k-1}(0;z)\mathcal{Y}_{k}(0;z)]}
      {\mathcal{Y}_{k}(0;z)[\mathcal{Y}_{k}(0;z)-h(k;z)\mathcal{Y}_{k-1}(0;z)]}\\
&=
\dfrac{h(k;z)z^{k-1}\prod_{j=1}^{k-1}(1-g_{j-1})g_j}
      {\mathcal{Y}_{k}(0;z)[\mathcal{Y}_{k}(0;z)-h(k;z)\mathcal{Y}_{k-1}(0;z)]},
\end{align*}
where the last equality follows from \cite[eqn.(2.1.9)]{Jones-Thron-book}.
Denoting $d_j=(1-g_{j-1})g_j$, $j\geq1$, we have from
\eqref{eqn:tail sequence for g fraction with gk missing and terminating with hkz}
\begin{align*}
\dfrac{\mathcal{X}_{k}(h(k;z);z)}{\mathcal{Y}_{k}(h(k;z);z)}=
\dfrac{\mathcal{X}_{k}(0;z)}{\mathcal{Y}_{k}(0;z)}-
                 \dfrac{\prod_{j=1}^{k-1}d_jz^{k-1}h(k;z)}
                       {\mathcal{Y}_{k-1}(0;z)\mathcal{Y}_{k}(0;z)h(k;z)-
                       [\mathcal{Y}_{k}(0;z)]^2}.
\end{align*}
\begin{note}
In the sequel, by $\mathcal{F}(z)$ we will mean the
unperturbed $g$-fraction as given in
\eqref{eqn:definition of g-fraction expansion}
with $g_k\in[0,1]$, $k\in\mathbb{Z}_{+}$.
Further, as the notation suggests, the rational function
$\mathcal{S}_k(0;z)$ is independent of the parameter $g_k$
and is known whenever $\mathcal{F}(z)$ is given.
The information of the missing parameter $g_k$
at the $k^{th}$ position
is stored in $h(k;z)$ and hence the notation
$\mathcal{F}(k;z)$.
\end{note}
It may also be noted that the polynomials
$\mathcal{Y}_{k}(0;z)$
can be easily computed from the Wallis recurrence
\cite[eqn. (2.1.6)]{Jones-Thron-book}
\begin{align*}
\mathcal{Y}_{j}(0;z)=\mathcal{Y}_{j-1}(0;z)-
                     (1-g_{j-2})g_{j-1}z\mathcal{Y}_{j-2}(0;z),
\quad j\geq2,
\end{align*}
with the initial values $\mathcal{Y}_{0}(0;z)=\mathcal{Y}_{1}(0;z)=1$.
Thus, we state our first result.
\begin{theorem}
\label{thm:structural relation for gk missing}
Suppose $\mathcal{F}(z)$ is given.
Let $\mathcal{F}(k;z)$ denote the
perturbed $g$-fraction in which the parameter $g_k$ is missing. Then, with
$d_j=(1-g_{j-1})g_j$, $j\geq1$
\begin{align}
\label{eqn:structural relation for gk missing}
\mathcal{F}(k;z)=\mathcal{S}_k(0;z)-
                 \dfrac{\prod_{j=1}^{k-1}d_jz^{k-1}h(k;z)}
                       {\mathcal{Y}_{k-1}(0;z)\mathcal{Y}_{k}(0;z)h(k;z)-[\mathcal{Y}_{k}(0;z)]^2},
\end{align}
where $\mathcal{Y}_k(0;z)$, $\mathcal{S}_k(0;z)$ and
$-(1-g_{k-1})^{-1}(1-g_k)h(k;z)$ are respectively,
the $k^{th}$ partial denominator,
the $k^{th}$ approximant and
the $(k+1)^{th}$ tail of $\mathcal{F}(z)$.
\end{theorem}
It may be observed that the right side of
\eqref{eqn:structural relation for gk missing}
is of the form
\begin{align*}
\dfrac{a(z)h(k;z)+b(z)}{c(z)h(k;z)+d(z)},
\end{align*}
with $a(z), b(z), c(z), d(z)$ being well defined polynomials.
Rational functions of such form are said to be
rational transformation of $h(k;z)$ and occur frequently in
perturbation theory of orthogonal polynomials.
For example, see \cite{Garza-Marcellan-szego-spectral-JCAM-2009,
Zhedanov-rational-spectral-JCAM-1997}.

A similar result for the perturbed $g$-fraction in which a finite number
of consecutive parameters are missing can be obtained by an analogous
argument that is stated directly as
\begin{theorem}
\label{thm:structural relation for finite number of parameters missing}
Let $\mathcal{F}(z)$ be given. Let
$\mathcal{F}(k,k+1,\cdots,k+l-1;z)$ denote the perturbed $g$-fraction in which
the l consecutive parameters $g_k,g_{k+1},\cdots,g_{k+l-1}$ are missing.
Then,
\begin{align}
\label{eqn:structural relation for finite number of parameters missing}
\lefteqn{\mathcal{F}(k,k+1,\cdots,k+l-1;z)=}\nonumber\\
&&\mathcal{S}_k(0;z)-\dfrac{\prod_{j=1}^{k-1}d_jz^{k-1}h(k,k+1,\cdots,k+l-1;z)}
                   {\mathcal{Y}_{k-1}(0;z)\mathcal{Y}_{k}(0;z)h(k,k+1,\cdots,k+l-1;z)-
                       [\mathcal{Y}_{k}(0;z)]^2},
\end{align}
where $-(1-g_{k-1})^{-1}(1-g_{k+l-1})h(k,k+1,\cdots,k+l-1;z)$
is the $(k+l)^{th}$ tail of $\mathcal{F}(z)$.
\end{theorem}
The next result is about the perturbation in which only two parameters
$g_k$ and $g_l$ are missing, where $l$ need not be $k\pm1$.
\begin{theorem}
\label{thm:structural relation for gk and gl missing}
Let $\mathcal{F}(z)$ be given.
Let $\mathcal{F}(k,l;z)$ denote the perturbed $g$-fraction
in which two parameters $g_k$ and $g_l$ are missing, where
we assume $l=k+m+1$, $m\geq1$.
Then
\begin{align}
\label{eqn:structural relation for gk and gl missing}
\mathcal{F}(k,l;z)=\mathcal{S}_k(0;z)-
                 \dfrac{\prod_{j=1}^{k-1}d_jz^{k-1}h(k,l;z)}
                       {\mathcal{Y}_{k-1}(0;z)\mathcal{Y}_{k}(0;z)h(k,l;z)-[\mathcal{Y}_{k}(0;z)]^2}
\end{align}
where $-(1-g_{k-1})^{-1}(1-g_k)h(k,l;z)$ is the perturbed $(k+1)^{th}$ tail of
$\mathcal{F}(z)$ in which $g_l$ is missing and is given by
\begin{align*}
\lefteqn{-(1-g_{k-1})^{-1}(1-g_k)h(k,k+m+1;z)=}\\
&&\mathcal{S}_{m}^{(k+1)}(0,z)-
\dfrac{\prod_{j=k+1}^{k+m}d_jz^m h(k+m+1;z)}
{[\mathcal{Y}^{(k+1)}_{m}(0;z)]^2-
\mathcal{Y}^{(k+1)}_{m-1}(0;z)\mathcal{Y}^{(k+1)}_{m}(0;z)h(k+m+1;z)}
\end{align*}
where $\mathcal{Y}^{(k+1)}_{m}(0;z)$ and $\mathcal{S}_{m}^{(k+1)}(0,z)$
are respectively, the
$m^{th}$ partial denominator and
$m^{th}$ approximant of the
$(k+1)^{th}$ tail of $\mathcal{F}(z)$.
Here, $-(1-g_{l-1})^{-1}(1-g_l)h(l;z)$ is the $(l+1)^{th}$ tail
of $\mathcal{F}(z)$.
\end{theorem}
\begin{proof}
Let
\begin{align}
\label{eqn:perturbed tail sequence with gl missing}
\mathcal{H}_{k+1}(l;z)=
\dfrac{g_{k+1}z}{1}
\begin{array}{cc}\\$-$\end{array}
\dfrac{(1-g_{k+1})g_{k+2}z}{1}
\begin{array}{cc}\\$-$\end{array}
\cdots
\dfrac{(1-g_{l-1})g_{l+1}z}{1}
\begin{array}{cc}\\$-$\end{array}
\dfrac{(1-g_{l+1})g_{l+2}z}{1}
\begin{array}{cc}\\$-$\end{array}
\cdots
\end{align}
so that $-(1-g_k)\mathcal{H}_{k+1}(l;z)$ is the perturbed
$(k+1)^{th}$ tail of $\mathcal{F}(z)$ in which $g_l$ is missing.
Then we can write
\begin{align*}
\mathcal{F}(k,l;z)=
\dfrac{\mathcal{X}_{k}(h(k,l;z);z)}{\mathcal{Y}_{k}(h(k,l;z);z)}=
\dfrac{1}{1}
\begin{array}{cc}\\$-$\end{array}
\dfrac{(1-g_0)g_1z}{1}
\begin{array}{cc}\\$-$\end{array}
\cdots
\dfrac{(1-g_{k-2})g_{k-1}z}{1-h(k,l;z)},
\end{align*}
where $h(k,l;z)=(1-g_{k-1})\mathcal{H}_{k+1}(l;z)$.
Now, proceeding as in
Theorem~$\ref{thm:structural relation for gk missing}$,
we obtain
\begin{align*}
\mathcal{F}(k,l;z)=\mathcal{S}_k(0;z)-
                 \dfrac{\prod_{j=1}^{k-1}d_jz^{k-1}h(k,l;z)}
                       {\mathcal{Y}_{k-1}(0;z)\mathcal{Y}_{k}(0;z)h(k,l;z)-[\mathcal{Y}_{k}(0;z)]^2}
\end{align*}
Hence all that remains is to find the expression for $h(k,l;z)$
or $\mathcal{H}_{k+1}(l;z)$.

Now, let
\begin{align*}
\mathcal{H}_{l+1}(z)=
\dfrac{g_{l+1}z}{1}
\begin{array}{cc}\\$-$\end{array}
\dfrac{(1-g_{l+1})g_{l+2}z}{1}
\begin{array}{cc}\\$-$\end{array}
\dfrac{(1-g_{l+2})g_{l+3}z}{1}
\begin{array}{cc}\\$-$\end{array}
\cdots,
\end{align*}
and $h(l;z)=(1-g_{l-1})\mathcal{H}_{l+1}(z)$.
From
\eqref{eqn:perturbed tail sequence with gl missing}
and
\cite[eqn.(1.1.4), p.57]{Lisa-Waadeland-book-cf-with-application},
we have
\begin{align*}
-(1-g_k)\mathcal{H}_{k+1}(l;z)
&=
\dfrac{-(1-g_k)g_{k+1}z}{1}
\begin{array}{cc}\\$-$\end{array}
\dfrac{(1-g_{k+1})g_{k+2}z}{1}
\begin{array}{cc}\\$-$\end{array}
\cdots
\dfrac{(1-g_{l-2})g_{l-1}z}{1-h(l;z)}\\
&=
\dfrac{\mathcal{X}^{(k+1)}_{l-k-1}(h(l;z);z)}{\mathcal{Y}^{(k+1)}_{l-k-1}(h(l;z);z)}.
\end{align*}
It is clear that, the rational function
$[\mathcal{X}^{(k+1)}_{l-k-1}(0;z)/\mathcal{Y}^{(k+1)}_{l-k-1}(0;z)]$
is the approximant of the $(k+1)^{th}$ tail
$-(1-g_k)\mathcal{H}_{k+1}(z)$
of $\mathcal{F}(z)$.
Then, using \cite[eqn. (1.1.6), p. 57]{Lisa-Waadeland-book-cf-with-application}
we obtain
\begin{align*}
\dfrac{\mathcal{X}^{(k+1)}_{l-k-1}(h(l;z);z)}{\mathcal{Y}^{(k+1)}_{l-k-1}(h(l;z);z)}-
\dfrac{\mathcal{X}^{(k+1)}_{l-k-1}(0;z)}{\mathcal{Y}^{(k+1)}_{l-k-1}(0;z)}=
\dfrac{-h(l;z)\prod_{k}^{l-1}[d_jz]}
{\mathcal{Y}^{(k+1)}_{l-k-1}(0;z)[\mathcal{Y}^{(k+1)}_{l-k-1}(0;z)-h(l;z)\mathcal{Y}^{(k+1)}_{l-k-2}(0;z)]}
\end{align*}
Finally, using the fact that $l=k+m+1$, we obtain
\begin{align*}
\lefteqn{-(1-g_k)\mathcal{H}_{k+1}(k+m+1;z)=}\\
&&\mathcal{S}_{m}^{(k+1)}(0,z)-
\dfrac{\prod_{j=k+1}^{k+m}d_jz^m h(k+m+1;z)}
{[\mathcal{Y}^{(k+1)}_{k+m+1}(0;z)]^2-\mathcal{Y}^{(k+1)}_{k+m}(0;z)\mathcal{Y}^{(k+1)}_{k+m+1}(0;z)h(k+m+1;z)},
\end{align*}
where
\begin{align*}
\mathcal{S}_{m}^{(k+1)}(0,z)=
\dfrac{-(1-g_k)g_{k+1}z}{1}
\begin{array}{cc}\\$-$\end{array}
\dfrac{(1-g_{k+1})g_{k+2}z}{1}
\begin{array}{cc}\\$-$\end{array}
\cdots
\dfrac{(1-g_{k+m-1})g_{k+m}z}{1}
\end{align*}
is the $m^{th}$ approximant of the $(k+1)^{th}$ tail of $\mathcal{F}(z)$.
\end{proof}
As mentioned earlier, from
\eqref{eqn:structural relation for gk missing},
\eqref{eqn:structural relation for finite number of parameters missing} and
\eqref{eqn:structural relation for gk and gl missing},
it is clear that tail sequences play a
significant role in deriving the
structural relations for the gap $g$-fractions.
We now illustrate the role of tail
sequences
using particular $g$-fraction expansions.
\subsection{Tail sequences using hypergeometric functions}
The Gaussian hypergeometric function, with the
complex parameters $a$, $b$ and $c$ is defined by the power series
\begin{align*}
F(a,b;c;\omega)=
\sum_{n=0}^{\infty}\dfrac{(a)_n(b)_n}{(c)_n(1)_n}\omega^n,
\quad |\omega|<1
\end{align*}
where $c\neq0,-1,-2,\cdots$ and
$(a)_n=a(a+1)\cdots(a+n-1)$ is the Pochhammer symbol.

Two hypergeometric functions
$F(a_1,b_1;c_1;\omega)$ and $F(a_2,b_2;c_2,\omega)$
are said to be contiguous if the difference between the corresponding parameters
is at most unity. A linear combination of two contiguous hypergeometric
functions is again a hypergeometric function. Such relations are called contiguous relations
and have been used to explore many hidden properties of
the hypergeometric functions; for example, many special functions
can be represented by ratios of Gaussian hypergeometric functions.
For more details, we refer to \cite{Andrews-Askey-Roy-book}.

Consider the Gauss continued fraction \cite[p. 337]{Wall_book}
(with $b\mapsto b-1$ and $c\mapsto c-1$)
\begin{align}
\label{eqn:Gauss continued fraction}
\dfrac{F(a,b;c;\omega)}{F(a,b-1;c-1;\omega)}=
\dfrac{1}{1}
\begin{array}{cc}\\$-$\end{array}
\dfrac{(1-g_0)g_1\omega}{1}
\begin{array}{cc}\\$-$\end{array}
\dfrac{(1-g_1)g_2\omega}{1}
\begin{array}{cc}\\$-$\end{array}
\dfrac{(1-g_2)g_3\omega}{1}
\begin{array}{cc}\\$-$\end{array}
\cdots
\end{align}
where
\begin{align*}
g_{2p}=\dfrac{c-a+p-1}{c+2p-1},
\quad
g_{2p+1}=\dfrac{c-b+p}{c+2p},
\quad p\geq0.
\end{align*}
Let $k_p=1-g_p$, $p\geq0$. We aim to find the ratio of
hypergeometric functions given by the continued fraction
\begin{align*}
\dfrac{1}{1}
\begin{array}{cc}\\$-$\end{array}
\dfrac{k_1\omega}{1}
\begin{array}{cc}\\$-$\end{array}
\dfrac{(1-k_1)k_2\omega}{1}
\begin{array}{cc}\\$-$\end{array}
\dfrac{(1-k_2)k_3\omega}{1}
\begin{array}{cc}\\$-$\end{array}
\cdots
\end{align*}
For this, first note that from
\eqref{eqn:Gauss continued fraction},
we can write
\begin{align*}
\mathcal{R}(\omega)=
1-\dfrac{1}{k_0}\left[1-\dfrac{F(a,b-1;c-1;\omega)}{F(a,b;c;\omega)}\right]=
1-\dfrac{(1-k_1)\omega}{1-\dfrac{k_1(1-k_2)\omega}{1-\dfrac{k_2(1-k_3)\omega}
{1-\cdots}}}
\end{align*}
Now, replacing $b\mapsto b-1$ and $c\mapsto c-1$ in
the contiguous relation
\cite[eqn. 89.6, p. 336]{Wall_book}
we obtain
\begin{align}
\label{eqn:contiguous relation used in derivation of gauss cf}
F(a,b;c;\omega)-F(a,b-1;c-1;\omega)=
\dfrac{a(c-b)}{(c-1)c}\omega F(a+1,b;c+1;\omega)
\end{align}
Hence, with $k_0=(c-a-1)/(c-1)$, we have
\begin{align*}
\mathcal{R}(\omega)
&=1-\dfrac{c-1}{a}
\left[\dfrac{F(a,b;c;\omega)-F(a,b-1;c-1;\omega)}{F(a,b;c;\omega)}\right]\\
&=1-\dfrac{c-b}{c}\dfrac{F(a+1,b;c+1;\omega)}{F(a,b;c;\omega)}\\
&=(1-\omega)\dfrac{F(a+1,b;c;\omega)}{F(a,b;c;\omega)},
\end{align*}
where the last equality follows from the relation
\begin{align}
\label{eqn:contiguous relation gauss to kustner}
F(a,b;c;\omega)=(1-\omega)F(a+1,b;c;\omega)+\dfrac{c-b}{c}\omega F(a+1,b;c+1;\omega),
\end{align}
which is easily proved by comparing the coefficients of
$\omega^k$ on both sides.
Finally, using the well known result
\cite[eqn. 75.3, p. 281]{Wall_book}, we obtain
\begin{align}
\label{eqn:Kustner continued fraction}
\dfrac{F(a+1,b;c;\omega)}{F(a,b;c;\omega)}=
\dfrac{1}{1}
\begin{array}{cc}\\$-$\end{array}
\dfrac{\frac{b}{c}\omega}{1}
\begin{array}{cc}\\$-$\end{array}
\dfrac{\frac{(c-b)(a+1)}{c(c+1)}\omega}{1}
\begin{array}{cc}\\$-$\end{array}
\dfrac{\frac{(c-a)(b+1)}{(c+1)(c+2)}\omega}{1}
\begin{array}{cc}\\$-$\end{array}
\cdots
\end{align}
Note that the continued fraction
\eqref{eqn:Kustner continued fraction}
has also been derived by different means in
\cite{Kustner-g-fractions-2002-CMFT} and studied
in the context of geometric properties of
hypergeometric functions.

For further analysis, we establish the following
formal continued fraction expansion:
\begin{align}
\label{eqn:gap g fraction identity starting g1}
\dfrac{F(a+1,b;c+1;\omega)}{F(a,b;c;\omega)}=
\dfrac{1}{1}
\begin{array}{cc}\\$-$\end{array}
\dfrac{\frac{b(c-a)}{c(c+1)}\omega}{1}
\begin{array}{cc}\\$-$\end{array}
\dfrac{\frac{(a+1)(c-b+1)}{(c+1)(c+2)}\omega}{1}
\begin{array}{cc}\\$-$\end{array}
\dfrac{\frac{(b+1)(c-a+1)}{(c+2)(c+3)}\omega}{1}
\begin{array}{cc}\\$-$\end{array}
\cdots
\end{align}
Suppose for now, the left hand side of
\eqref{eqn:gap g fraction identity starting g1}
is denoted by $\mathcal{G}_1^{(a,b,c)}(\omega)$. Again using, Gauss
continued fraction \cite[eqn. 89.9, p. 337]{Wall_book}
with $a\mapsto a+1$ and $c\mapsto c+1$,
we arrive at
\begin{align*}
[\mathcal{G}_1^{(a,b,c)}]^{-1}(\omega)=1-\dfrac{b(c-a)}{c(c+1)}\omega
\dfrac{F(a+1,b+1;c+2;\omega)}{F(a+1,b;c+1;\omega)}.
\end{align*}
In the relation
\eqref{eqn:contiguous relation used in derivation of gauss cf},
interchanging $a$ and $b$ and then replacing $a\mapsto a+1$ and
$c\mapsto c+1$, we get
\begin{align*}
1-\dfrac{b(c-a)}{c(c+1)}\omega\dfrac{F(a+1,b+1;c+2;\omega)}{F(a+1,b;c+1;\omega)}=
\dfrac{F(a,b;c;\omega)}{F(a+1,b;c+1;\omega)},
\end{align*}
which implies
\begin{align*}
\mathcal{G}_1^{(a,b,c)}(\omega)=
\dfrac{F(a+1,b;c+1;\omega)}{F(a,b;c;\omega)}.
\end{align*}
As mentioned, the right hand side is only a
formal expansion for the left hand side in
\eqref{eqn:gap g fraction identity starting g1}.
However, note the fact that the sequence $\{\mathcal{P}_j(\omega)\}_{j=0}^{\infty}$,
\begin{align*}
\mathcal{P}_{2j}(\omega)&=F(a+j,b+j;c+2j;\omega),\\
\mathcal{P}_{2j+1}(\omega)&=F(a+j+1,b+j;c+2j+1;\omega),
\quad j\geq0
\end{align*}
satisfies the difference equation
\begin{align}
\label{eqn:difference relation for correspondence}
\mathcal{P}_{j}(\omega)=\mathcal{P}_{j+1}(\omega)-
d_{j+1}\omega\mathcal{P}_{j+2}(\omega),
\quad j\geq0,
\end{align}
where
\begin{align*}
d_n=
\left\{
  \begin{array}{ll}
    \dfrac{(b+j)(c-a+j)}{(c+2j)(c+2j+1)}, & \hbox{$n=2j+1\geq1$, $j\geq0$;} \\
    \dfrac{(a+j)(c-b+j)}{(c+2j-1)(c+2j)}, & \hbox{$n=2j\geq2$, $j\geq1$.}
  \end{array}
\right.
\end{align*}
Thus, using
\cite[Theorem 1, pp. 295-296]{Lisa-Waadeland-book-cf-with-application},
we can conclude that the right side of
\eqref{eqn:gap g fraction identity starting g1}
indeed corresponds as well as converges to the
left side of
\eqref{eqn:gap g fraction identity starting g1},
which we state as
\begin{proposition}
\label{prop:correspondence and convergence of g1}
  The following correspondence and convergence properties hold.
\begin{enumerate}[(i)]
  \item With $a$, $b$, $c\neq0,-1,-2,\cdots$ complex constants,
        \begin{align}
        \label{eqn:continued fraction used in convergence correspondence prop}
         \dfrac{F(a+1,b;c+1;\omega)}{F(a,b;c;\omega)}\sim
         \dfrac{1}{1}
         \begin{array}{cc}\\$-$\end{array}
         \dfrac{\frac{(b)(c-a)}{c(c+1)}\omega}{1}
         \begin{array}{cc}\\$-$\end{array}
         \dfrac{\frac{(a+1)(c-b+1)}{(c+1)(c+2)}\omega}{1}
         \begin{array}{cc}\\$-$\end{array}
         \dfrac{\frac{(b+1)(c-a+1)}{(c+2)(c+3)}\omega}{1}
         \begin{array}{cc}\\$-$\end{array}
         \cdots
        \end{align}
  \item The continued fraction on the right side of
        \eqref{eqn:gap g fraction identity starting g1}
        converges to the meromorphic function $f(\omega)$ in the cut-plane $\mathfrak{D}$
        where $f(\omega)=F(a+1,b;c+1;\omega)/F(a,b;c;\omega)$ and
        $\mathfrak{D}=\{\omega\in\mathbb{C}:|\arg(1-\omega)|<\pi\}$. The convergence is uniform
        on every compact subset of $\{\omega\in\mathfrak{D}:f(\omega)\neq \infty\}$.
\end{enumerate}
\end{proposition}
We would like to mention here that the polynomial sequence
$\{\mathcal{Q}_n(\omega)\}$
corresponding to $\{\mathcal{P}_n(\omega)\}$
and that arises during the discussion of the convergence of
Gauss continued fractions
\cite[p. 294]{Lisa-Waadeland-book-cf-with-application}
is given by

\begin{align*}
\mathcal{Q}_{2j}(\omega)&=F(a+j,b+j;c+2j;\omega), \quad j\geq0 \\
\mathcal{Q}_{2j+1}(\omega)&=F(a+j,b+j+1;c+2j+1;\omega), \quad j\geq0.
\end{align*}
Also note that the continued fraction used in the right side of
\eqref{eqn:continued fraction used in convergence correspondence prop}
is
\begin{align*}
\dfrac{1}{1}
\begin{array}{cc}\\$-$\end{array}
\dfrac{k_1(1-k_{2})\omega}{1}
\begin{array}{cc}\\$-$\end{array}
\dfrac{k_{2}(1-k_{3})\omega}{1}
\begin{array}{cc}\\$-$\end{array}
\dfrac{k_{3}(1-k_{4})\omega}{1}
\begin{array}{cc}\\$-$\end{array}
\cdots.
\end{align*}
Here, we recall that $k_n=1-g_n$, $n\geq0$, where $\{g_n\}$ are the parameters appearing in the Gauss continued fraction
\eqref{eqn:Gauss continued fraction}.
The following result gives a kind of generalization of
Proposition $\ref{prop:correspondence and convergence of g1}$.
The correspondence and convergence properties of the
continued fractions involved can be discussed similar to the one for
$\mathcal{G}_1^{(a,b,c)}(\omega)$.
\begin{proposition}
\label{prop:continued fraction identity for gk}
Let,
\begin{align}
\label{eqn:gap g fraction identity starting gk}
\mathcal{G}_{n}^{(a,b,c)}(\omega)=
\dfrac{1}{1}
\begin{array}{cc}\\$-$\end{array}
\dfrac{k_n(1-k_{n+1})\omega}{1}
\begin{array}{cc}\\$-$\end{array}
\dfrac{k_{n+1}(1-k_{n+2})\omega}{1}
\begin{array}{cc}\\$-$\end{array}
\dfrac{k_{n+2}(1-k_{n+3})\omega}{1}
\begin{array}{cc}\\$-$\end{array}
\cdots
\end{align}
Then,
\begin{align*}
\mathcal{G}_{2j}^{(a,b,c)}(\omega)&=
\dfrac{F(a+j,b+j;c+2j;\omega)}{F(a+j,b+j-1;c+2j-1;\omega)}
\quad j\geq1,\\
\mathcal{G}_{2j+1}^{(a,b,c)}(\omega)&=
\dfrac{F(a+j+1,b+j;c+2j+1;\omega)}{F(a+j,b+j;c+2j;\omega)}
\quad j\geq0
\end{align*}
\end{proposition}
\begin{proof}
The case $j=0$ has already been established in
Proposition $\ref{prop:correspondence and convergence of g1}$.

Comparing the continued fractions for
$\mathcal{G}_{2j+1}^{(a,b,c)}(\omega)$
and $\mathcal{G}_{2j-1}^{(a,b,c)}(\omega)$,
$j\geq1$,
it can be seen that
$\mathcal{G}_{2j+1}^{(a,b,c)}(\omega)$
can be obtained for
$\mathcal{G}_{2j-1}^{(a,b,c)}(\omega)$ $j\geq1$
by shifting
$a\mapsto a+1$, $b\mapsto b+1$ and $c\mapsto c+2$.

For $n=2j$, $j\geq1$,
we note that the continued fraction in right side of
\eqref{eqn:gap g fraction identity starting gk}
is nothing but the Gauss continued fraction
\cite[eqn. 89.9, p. 337]{Wall_book}
with the shifts $a\mapsto a+j$, $b\mapsto b+j-1$
and $c\mapsto c+2j-1$
in the parameters.
\end{proof}
Instead of starting with $k_n(1-k_{n+1})$, as the first partial numerator
term in the continued fraction
\eqref{eqn:gap g fraction identity starting gk},
a modification by inserting a new term changes the hypergeometric ratio given in Proposition~\ref{prop:continued fraction identity for gk}, thus leading to interesting consequences. We state this result as follows.
\begin{theorem}
\label{thm:gap g fraction identity with minimal parameters}
Let
\begin{align}
\label{eqn:k-fraction expansion for Fn(abc)}
\mathcal{F}_n^{(a,b,c)}(\omega)=
\dfrac{1}{1}
\begin{array}{cc}\\$-$\end{array}
\dfrac{k_n\omega}{1}
\begin{array}{cc}\\$-$\end{array}
\dfrac{(1-k_{n})k_{n+1}\omega}{1}
\begin{array}{cc}\\$-$\end{array}
\dfrac{(1-k_{n+1})k_{n+2}\omega}{1}
\begin{array}{cc}\\$-$\end{array}
\cdots .
\end{align}
Then,
\begin{align*}
\mathcal{F}_{2j+1}^{(a,b,c)}(\omega)
&=\dfrac{F(a+j+1,b+j;c+2j;\omega)}{F(a+j,b+j;c+2j;\omega)},
\quad j\geq0,\\
\mathcal{F}_{2j+2}^{(a,b,c)}(\omega)
&=\dfrac{F(a+j+1,b+j+1;c+2j+1;\omega)}{F(a+j+1,b+j;c+2j+1;\omega)},
\quad j\geq0.
\end{align*}
\end{theorem}
\begin{proof}
Denoting,
\begin{align*}
\mathcal{E}_{n+1}^{(a,b,c)}(\omega)=
1-\dfrac{1}{k_n}\left(1-\dfrac{1}{\mathcal{G}_n^{(a,b,c)}(\omega)}\right)=
1-\dfrac{(1-k_{n+1})\omega}{1-\dfrac{k_{n+1}(1-k_{n+2})\omega}{1-\cdots}},
\quad n\geq1,
\end{align*}
we find from \cite[eqn. 75.3, p. 281]{Wall_book},
\begin{align*}
\mathcal{F}_{n+1}^{(a,b,c)}(\omega)=
\dfrac{\mathcal{E}_{n+1}^{(a,b,c)}(\omega)}{1-z}=
\dfrac{1}{1-\dfrac{k_{n+1}\omega}{1-\dfrac{(1-k_{n+1})k_{n+2}\omega}{1-\cdots}}},
\quad n\geq1.
\end{align*}
Hence, we need to derive the functions
$\mathcal{E}_{n+1}^{(a,b,c)}(\omega)$. For $n=2j$, $j\geq1$,
using \eqref{eqn:contiguous relation used in derivation of gauss cf}
and $k_{2j}=(a+j)/(c+2j-1)$, we find that
\begin{align*}
\dfrac{1}{k_{2j}}\left[1-\dfrac{1}{\mathcal{G}_{2j}^{(a,b,c)}(\omega)}\right]=
\dfrac{c-b+j}{c+2j}\omega\dfrac{F(a+j+1,b+j;c+2j+1;\omega)}
{F(a+j,b+j;c+2j;\omega)}
\end{align*}
Shifting $a\mapsto a+j$, $b\mapsto b+j$ and $c\mapsto c+2j$ in
\eqref{eqn:contiguous relation gauss to kustner},
we find that
\begin{align*}
\mathcal{E}_{2j+1}^{(a,b,c)}(\omega)=(1-z)\dfrac{F(a+j+1,b+j;c+2j;\omega)}
{F(a+j,b+j;c+2j;\omega)}
\end{align*}
so that
\begin{align*}
\mathcal{F}_{2j+1}^{(a,b,c)}(\omega)=
\dfrac{F(a+j+1,b+j;c+2j;\omega)}{F(a+j;b+j;c+2j;\omega)}=
\dfrac{1}{1-\dfrac{k_{2j+1}\omega}
{1-\dfrac{(1-k_{2j+1})k_{2j+2}\omega}{1-\cdots}}},
\quad j\geq1.
\end{align*}
Repeating the above steps, we find that for $n=2j+1$, $j\geq0$ and
$k_{2j+1}=(b+j)/(c+2j)$, $j\geq0$,
\begin{align*}
\mathcal{E}_{2j+2}^{(a,b,c)}(\omega)=(1-z)\dfrac{F(a+j+1,b+j+1;c+2j+1;\omega)}
{F(a+j+1,b+j;c+2j+1;\omega)},
\quad j\geq0
\end{align*}
so that
\begin{align*}
\mathcal{F}_{2j+2}^{(a,b,c)}(\omega)=
\dfrac{F(a+j+1,b+j+1;c+2j+1;\omega)}{F(a+j+1;b+j;c+2j+1;\omega)}=
\dfrac{1}{1-\dfrac{k_{2j+2}\omega}{1-\dfrac{(1-k_{2j+2})k_{2j+3}\omega}{1-\cdots}}},
\quad j\geq0.
\end{align*}
\end{proof}
%
%
For particular values of $\mathcal{F}_n^{(a,b,c)}(\omega)$, further properties of the ratio of hypergeometric function can be discussed. One particular case and few ratios of hypergeometric functions are given in
Section~\ref{sec:a class of pick functions and schur functions}
with some properties. Before proving such specific case, we consider another type of perturbation in $g$-fraction in the next section.
\section{Perturbed Schur parameters}
\label{sec:Perturbed Schur parameters}
As mentioned in Section~\ref{sec:Introduction},
the case of a single parameter
$g_k$ being replaced by $g_{k}^{(\beta_k)}$ can be studied using the
Schur parameters. It is obvious that this is equivalent to studying the
perturbed sequence $\{\alpha_{j}^{(\beta_k)}\}_{j=0}^{\infty}$,
where
\begin{align}
\label{eqn:relation between schur paramters and perturbed ones}
\alpha_{j}^{(\beta_k)}=
\left\{
  \begin{array}{ll}
    \alpha_j, & \hbox{$j\neq k$;} \\
    \beta_k, & \hbox{$j=k$.}
  \end{array}
\right.
\end{align}
Hence, we start with a given Schur function and study the perturbed
Carath\'{e}odory function and its corresponding $g$-fraction.
The following theorem gives the structural relation between the Schur
function and the perturbed one. The proof follows the transfer matrix
approach, which has also been used earlier
in literature (see for example
\cite{Castillo-pert-szego-rec-2014-jmaa,
Castillo-copolynomials-2015-jmaa}).
\begin{theorem}
\label{thm:structural relations for perturbations}
Let $\mathrm{A}_k(z)$ and $B_k(z)$ be the $n^{th}$ partial numerators
and denominators of the Schur fraction associated with the sequence
$\{\alpha_k\}_{n=0}^{\infty}$. If $\mathrm{A}_{k}(z;k)$
and $\mathrm{B}_{k}(z;k)$ are the $n^{th}$ partial numerators
and denominators of the Schur fraction associated with the sequence
$\{\alpha_j^{(\beta_k)}\}_{j=0}^{\infty}$ as defined in
\eqref{eqn:relation between schur paramters and perturbed ones},
then the following structural relations hold for
$p\geq 2k$, $k\geq1$.
\begin{align}
\label{eqn:structural relations used in theorem}
z^{k-1}\prod_{j=0}^{k}(1-|\alpha_j|^2)
\left(
  \begin{array}{cc}
    \mathrm{A}_{2p+1}(z;k) & \mathrm{A}_{2p}(z;k) \\
    \mathrm{B}_{2p+1}(z;k) & \mathrm{B}_{2p}(z;k) \\
  \end{array}
\right)=
\mathfrak{T}(z;k)
\left(
  \begin{array}{cc}
    \mathrm{A}_{2p+1}(z) & \mathrm{A}_{2p}(z) \\
    \mathrm{B}_{2p+1}(z) & \mathrm{B}_{2p}(z) \\
  \end{array}
\right),
\end{align}
where the entries of the transfer matrix
$\mathfrak{T}(z;k)$ are given by
\begin{align*}
&\left(
  \begin{array}{cc}
   \mathfrak{T}_{(1,1)} & \mathfrak{T}_{(1,2)} \\
   \mathfrak{T}_{(2,1)} & \mathfrak{T}_{(2,2)} \\
  \end{array}
  \right)\\
=&\left(
   \begin{array}{cc}
    p_k(z,k)A_{2k-1}(z)+q_k^{\ast}(z,k)A_{2k-2}(z)  & q_k(z,k)A_{2k-1}(z)+p_k^{\ast}(z,k)A_{2k-2}(z) \\
    p_k(z,k)B_{2k-1}(z)+q_k^{\ast}(z,k)B_{2k-2}(z)  & q_k(z,k)B_{2k-1}(z)+p_k^{\ast}(z,k)B_{2k-2}(z) \\
   \end{array}
   \right),
\end{align*}
with
\begin{align*}
p_k(z,k)&=(\alpha_k-\beta_k)B_{2k-1}(z)+(1-\beta\bar{\alpha}_k)B_{2k-2}(z)\\
q_k(z;k)&=(\beta_k-\alpha_k)A_{2k-1}(z)-(1-\bar{\alpha}_k\beta_k)A_{2k-2}(z).
\end{align*}
\end{theorem}
\begin{proof}
Let
\begin{align*}
\Omega_p(z;\alpha)=
\left(
  \begin{array}{cc}
    \mathrm{A}_{2p+1}(z) & \mathrm{B}_{2p+1}(z) \\
    \mathrm{A}_{2p}(z) & \mathrm{B}_{2p}(z) \\
  \end{array}
\right)
\quad\mbox{and}\quad
\Omega_p(z;\alpha;k)=
\left(
  \begin{array}{cc}
    \mathrm{A}_{2p+1}(z;k) & \mathrm{B}_{2p+1}(z;k) \\
    \mathrm{A}_{2p}(z;k) & \mathrm{B}_{2p}(z;k) \\
  \end{array}
\right).
\end{align*}
Then the matrix relation
\eqref{eqn:matrix relation for Schur approximants}
can be written as
\begin{align}
\label{eqn:transfer matrix before perturbation}
\Omega_p(z;\alpha)&=T_p(\alpha_p)\cdot\Omega_{p-1}(z;\alpha)\nonumber\\
                  &=T_p(\alpha_{p})\cdot T_{p-1}(\alpha_{p-1})\cdot\cdots\cdot
                   T_{1}(\alpha_{1})\cdot\Omega_0(z;\alpha),
\quad p\geq1,
\end{align}
with
\begin{align*}
T_p(\alpha_p)=
\left(
  \begin{array}{cc}
    z & \bar{\alpha}_pz \\
    \alpha_p & 1 \\
  \end{array}
\right)
\quad\mbox{and}\quad
\Omega_0(z;\alpha)=
\left(
  \begin{array}{cc}
    z & \bar{\alpha}_0z \\
    \alpha_0 & 1 \\
  \end{array}
\right)=
T_0(\alpha_0).
\end{align*}
From \eqref{eqn:relation between schur paramters and perturbed ones}, it is clear that
\begin{align}
\label{eqn:transfer matrix after perturbation}
\Omega_{p}(z;\alpha;k)=
\Omega_0(z;\alpha) T_{k}(\beta_{k})\prod_{\underset{j\neq k}{j=1}}^{p}T_j(\alpha_j).
\end{align}
Defining the associated polynomials of order $k+1$ as
\begin{align*}
\Omega_{p-(k+1)}^{(k+1)}(z;\alpha)=
T_p(\alpha_p)T_{p-1}(\alpha_{p-1})\cdots T_{k+1}(\alpha_{k+1})\Omega_0(z;\alpha),
\end{align*}
we have
\begin{align}
\label{eqn:polynomials of order k+1}
T_p(\alpha)T_{p-1}(\alpha_{p-1})\cdots T_{k+1}(\alpha_{k+1})=
\Omega_{p-(k+1)}^{(k+1)}(z;\alpha)[\Omega_0(z;\alpha)]^{-1},
\end{align}
where $[\Omega_0(z;\alpha)]^{-1}$ denotes the matrix inverse of
$\Omega_0(z;\alpha)$.
Now, using
\eqref{eqn:transfer matrix before perturbation}
and
\eqref{eqn:polynomials of order k+1}
in
\eqref{eqn:transfer matrix after perturbation},
we get
\begin{align}
\label{eqn:transfer matrix after perturbaton with polynomials of order k+1}
\Omega_p(z;k;\alpha)=\Omega_{p-(k+1)}^{(k+1)}(z;\alpha)\cdot[\Omega_{0}(z;\alpha)]^{-1}
                     \cdot T_k(\beta_k)\cdot\Omega_{k-1}(z;\alpha).
\end{align}
Again from \eqref{eqn:transfer matrix before perturbation},
\begin{align*}
\Omega_p(z;\alpha)&=\underbrace{T_p(\alpha_p)\cdots T_{k+1}(\alpha_{k+1})}\cdot
               \underbrace{T_k(\alpha_k)\cdots T_{1}(\alpha_{1})\Omega_0(z;\alpha)}\\
             &=\Omega_{p-(k+1)}^{(k+1)}(z;\alpha)\Omega_0^{-1}(z;\alpha)\cdot\Omega_k(z;\alpha),
\end{align*}
which means
\begin{align}
\label{eqn:transfer matrix for polynomials of order k+1}
\Omega_{p-(k+1)}^{(k+1)}(z;\alpha)=
\Omega_p(z;\alpha)[\Omega_{k}(z;\alpha)]^{-1}\Omega_0(z;\alpha).
\end{align}
Using
\eqref{eqn:transfer matrix for polynomials of order k+1}
in
\eqref{eqn:transfer matrix after perturbaton with polynomials of order k+1}, we get
\begin{align*}
\Omega_p(z;k;\alpha)=\Omega_p(z,\alpha)[\Omega_{k}(z,\alpha)]^{-1}
                     \Omega_0(z;\alpha)\cdot
                     [\Omega_0(z;\alpha)]^{-1}\cdot T_k(\beta_k)\cdot\Omega_{k-1}(z,\alpha),
\end{align*}
which implies
\begin{align*}
[\Omega_p(z;k;\alpha)]^{T}=[T_k(\beta_k)\Omega_{k-1}(z,\alpha)]^{T}
                         \cdot
                         [\Omega_{k}(z,\alpha)]^{-T}\cdot
                         [\Omega_p(z,\alpha)]^{T}.
 \end{align*}
where $[\Omega_p(z,\alpha)]^{T}$ denotes the matrix transpose of
$\Omega_p(z,\alpha)$.
After a brief calculation, and using the relations
\eqref{eqn:reciprocal relations between An and Bn},
it can be proved that the product
$[T_k(\beta_k)\Omega_{k-1}(z,\alpha)]^{T}\cdot
\Omega_{k}^{-T}(z,\alpha)$
precisely gives the transfer matrix
$\mathfrak{T}(z;k)$
leading to
\eqref{eqn:structural relations used in theorem}.
\end{proof}
As an important consequence of Theorem
\eqref{thm:structural relations for perturbations},
we have,
\begin{align*}
z^{k-1}\prod_{j=0}^{k}(1-|\alpha_j|^2)
\left(
  \begin{array}{c}
    \mathrm{A}_{2p}(z;k) \\
    \mathrm{B}_{2p}(z;k) \\
  \end{array}
\right)=
\mathfrak{T}_k(z;k)
\left(
  \begin{array}{c}
    \mathrm{A}_{2p}(z) \\
    \mathrm{B}_{2p}(z) \\
  \end{array}
\right),
\end{align*}
which implies,
\begin{align*}
\dfrac{\mathrm{A}_{2p}(z;k)}{\mathrm{B}_{2p}(z;k)}=
\dfrac{\mathfrak{T}_{(1,2)}+\mathfrak{T}_{(1,1)}\left(\mathrm{A}_{2p}(z)/\mathrm{B}_{2p}(z)\right)}
      {\mathfrak{T}_{(2,2)}+\mathfrak{T}_{(2,1)}\left(\mathrm{A}_{2p}(z)/\mathrm{B}_{2p}(z)\right)}.
\end{align*}
This gives the perturbed Schur function as,
\begin{align}
\label{eqn:perturbed Schur function}
f^{(\beta_k)}(z;k)=
\dfrac{\mathfrak{T}_{(1,2)}+\mathfrak{T}_{(1,1)}f(z)}
      {\mathfrak{T}_{(2,2)}+\mathfrak{T}_{(2,1)}f(z)}.
\end{align}
We next consider a non-constant Schur function of the  form $f(z)=cz+d$ where
$|c|+|d|\leq1$ and have a perturbation $\alpha_1\mapsto\beta_1$. Then
\begin{align*}
p_1(z,1)&=(\alpha_1-\beta_1)\bar{\alpha}_0z+(1-\bar{\alpha}_1\beta_1);
\quad
p_1^{\ast}(z,1)=(1-\alpha_1\bar{\beta}_1)z+(\bar{\alpha}_1-\bar{\beta}_1)\alpha_0;\\
q_1(z,1)&=(\beta_1-\alpha_1)z-(1-\bar{\alpha}_1\beta_1)\alpha_0;
\quad
q_1^{\ast}(z,1)=(\bar{\beta}_1-\bar{\alpha}_1)-(1-\alpha_1\bar{\beta}_1)\bar{\alpha}_0z.
\end{align*}
The matrix entries are
\begin{align*}
\tau_{(1,1)}&=(\alpha_1-\beta_1)z^2+[(1-\beta_1\bar{\alpha}_1)-(1-\alpha_1\bar{\beta}_1)|\alpha_0|^2]z+
  \alpha_0(\bar{\beta}_1-\bar{\alpha}_1),\\
\tau_{(1,2)}&=(\beta_1-\alpha_1)z^2+[(1-\alpha_1\bar{\beta}_1)\alpha_0-(1-\bar{\alpha}_1\beta_1)\alpha_0]z+
  (\bar{\alpha}_1-\bar{\beta}_1)\alpha_0^{2},\\
\tau_{(2,1)}&=(\alpha_1-\beta_1)(\bar{\alpha}_0)^2z^2+[(1-\beta_1\bar{\alpha}_1)\bar{\alpha}_0-(1-\alpha_1\bar{\beta}_1)\bar{\alpha}_0]z+
  (\bar{\beta}_1-\bar{\alpha}_1),\\
\tau_{(2,2)}&=(\beta_1-\alpha_1)\bar{\alpha}_0z^2+[(1-\alpha_1\bar{\beta}_1)-(1-\bar{\alpha}_1\beta_1)|\alpha_0|^2]z+
   (\bar{\alpha}_1-\bar{\beta}_1)\alpha_0.
\end{align*}
The transformed Schur function is a rational function given by
\begin{align*}
f^{(\beta_1)}(z,1)=
\dfrac{Az^3+Bz^2+Cz+D}{\hat{A}z^3+\hat{B}z^2+\hat{C}z+\hat{D}},
\end{align*}
where
\begin{align*}
A&=(\alpha_1-\beta_1)\bar{\alpha}_0c,
\quad
B=(\beta_1-\alpha_1)(1-\bar{\alpha}_0d)+
    c(1-\beta_1\bar{\alpha}_1)-
    c|\alpha_0|^2(1-\alpha_1\bar{\beta}_1),\\
C&=(1-\alpha_1\bar{\beta}_1)(\alpha_0-d|\alpha_0|^2)+
   (1-\bar{\alpha}_1\beta_1)(d-\alpha_0)+
   c\alpha_0(\bar{\beta}_1-\bar{\alpha}_1),\\
D&=(\bar{\beta}_1-\bar{\alpha}_1)(d-\alpha_0)\alpha_0,
\end{align*}
and
\begin{align*}
\hat{A}&=(\alpha_1-\beta_1)(\bar{\alpha}_0)^2c,
\quad
\hat{B}=(\beta_1-\alpha_1)(1-\bar{\alpha}_0d)\bar{\alpha}_0+
         c(1-\beta_1\bar{\alpha}_1)\bar{\alpha}_0-
         c(1-\alpha_1\bar{\beta}_1)\bar{\alpha}_0,\\
\hat{C}&=(1-\alpha_1\bar{\beta}_1)(1-d\bar{\alpha}_0)+
         (1-\beta_1\bar{\alpha}_1)(d\bar{\alpha}_0-|\alpha_0|^2)+
         c(\bar{\beta}_1-\bar{\alpha}_1),\\
\hat{D}&=(\bar{\beta}_1-\bar{\alpha}_1)(d-\alpha_0).
\end{align*}
This leads to the following easy consequence of Theorem
\ref{thm:structural relations for perturbations}.
\begin{corollary}
Let $f(z)=cz+\alpha_0$, where $|c|\leq 1-|\alpha_0|$
denote the class of Schur functions.
Then with the perturbation $\alpha_1\mapsto\beta_1$,
the resulting Schur function is the rational function given by
\begin{align}
\label{eqn:transformed Schur function with alpha1 replaced by beta1 used in proposition}
f^{(\beta_1)}(z;1)=\dfrac{Az^2+Bz+C}{\hat{A}z^2+\hat{B}z+\hat{C}},
\quad A\neq0,\quad \hat{A}\neq0.
\end{align}
\end{corollary}
We now consider an example illustrating the above discussion.
\begin{example}
\label{exm:perturbed schur fractions}
Consider the sequence of Schur parameters
$\{\alpha_n\}_{n=0}^{\infty}$
given by
$\alpha_0=1/2$ and $\alpha_n=2/(2n+1)$, $n\geq1$.
Then, as in \cite[Example 6.3]{JNT-Survey-Schur-PC-Szego},
the Schur function is $f(z)=(1+z)/2$
with
\begin{align*}
\mathrm{A}_{2m}(z)  &=\dfrac{1}{2}+\dfrac{2z^{m+2}-2(m+1)z^2+2mz}
                                         {(2m+1)(z-1)^2},\\
\mathrm{B}_{2m}(z)  &=1+\dfrac{z^{m+2}+z^{m+1}-(2m+1)z^2+(2m-1)z}
                              {(2m+1)(z-1)^2},\\
\mathrm{A}_{2m+1}(z)&=\dfrac{z+z^2-(2m+3)z^{m+2}+(2m+1)z^{m+3}}
                            {(2m+1)(z-1)^2},\\
\mathrm{B}_{2m+1}(z)&=\dfrac{z^{m+1}}{2}+2\dfrac{z-(m+1)z^{m+1}+mz^{m+2}}
                                                {(2m+1)(z-1)^2}.
\end{align*}
We study the perturbation $\alpha_1\mapsto\beta_1=1/2$.
For the transfer matrix $\mathfrak{T}(z;k)$,
the following polynomials are required.
\begin{align*}
p_1(z)&=\dfrac{z}{12}+\dfrac{2}{3};
\quad
p_1^{\ast}(z)=\dfrac{2}{3}z+\dfrac{1}{12};\\
q_1(z)&=-\dfrac{z}{6}-\dfrac{1}{3};
\quad
q_1^{\ast}(z)=-\dfrac{z}{3}-\dfrac{1}{6}.
\end{align*}
The entries of $\mathfrak{T}(z;k)$ are
\begin{align*}
\mathfrak{T}_{(1,1)}&=\dfrac{z^2}{12}+\dfrac{z}{2}-\dfrac{1}{12};
\quad
\mathfrak{T}_{(1,2)}=-\dfrac{z^2}{6}+\dfrac{1}{12};\\
\mathfrak{T}_{(2,1)}&=\dfrac{z^2}{24}-\dfrac{1}{6};
\quad
\mathfrak{T}_{(2,2)}=-\dfrac{z^2}{12}+\dfrac{z}{2}+\dfrac{1}{12}.
\end{align*}
Hence, the transformed Schur function using
\eqref{eqn:transformed Schur function with alpha1 replaced by beta1 used in proposition}
is
\begin{align}
\label{eqn:Transformed Schur function in example}
f^{(1/2)}(z;1)=2\dfrac{z^2-3z+5}{z^2-3z+20}.
\end{align}
Observe that $f(z)$ and $f^{(1/2)}(z;1)$ are analytic in $\mathbb{D}$ with $f(0)=f^{(1/2)}(0;1)$ and
\begin{align*}
\omega(z)=f^{-1}(f^{(1/2)}(z;1))=
\dfrac{3z(z-3)}{z^2-3z+20},
\end{align*}
where $\omega(z)$ is analytic in $\mathbb{D}$ with $|\omega(z)|<1$.
Further, by Schwarz lemma $|\omega(z)|<|z|$ for $0<|z|<1$ unless $\omega(z)$
is a pure rotation. In such a case the range of $f^{(1/2)}(z;1)$ is contained in the range of $f(z)$. The function $f^{(1/2)}(z;1)$ is said to be subordinate to $f(z)$ and written as $f^{(1/2)}(z;1)\prec f(z)$
for $z\in\mathbb{D}$
\cite[Chapter 6]{Duren-book}.

We plot the ranges of both the Schur functions below.
\pagebreak
\begin{center}
\begin{figure}[h!]
\includegraphics[width=5cm, height=5cm]{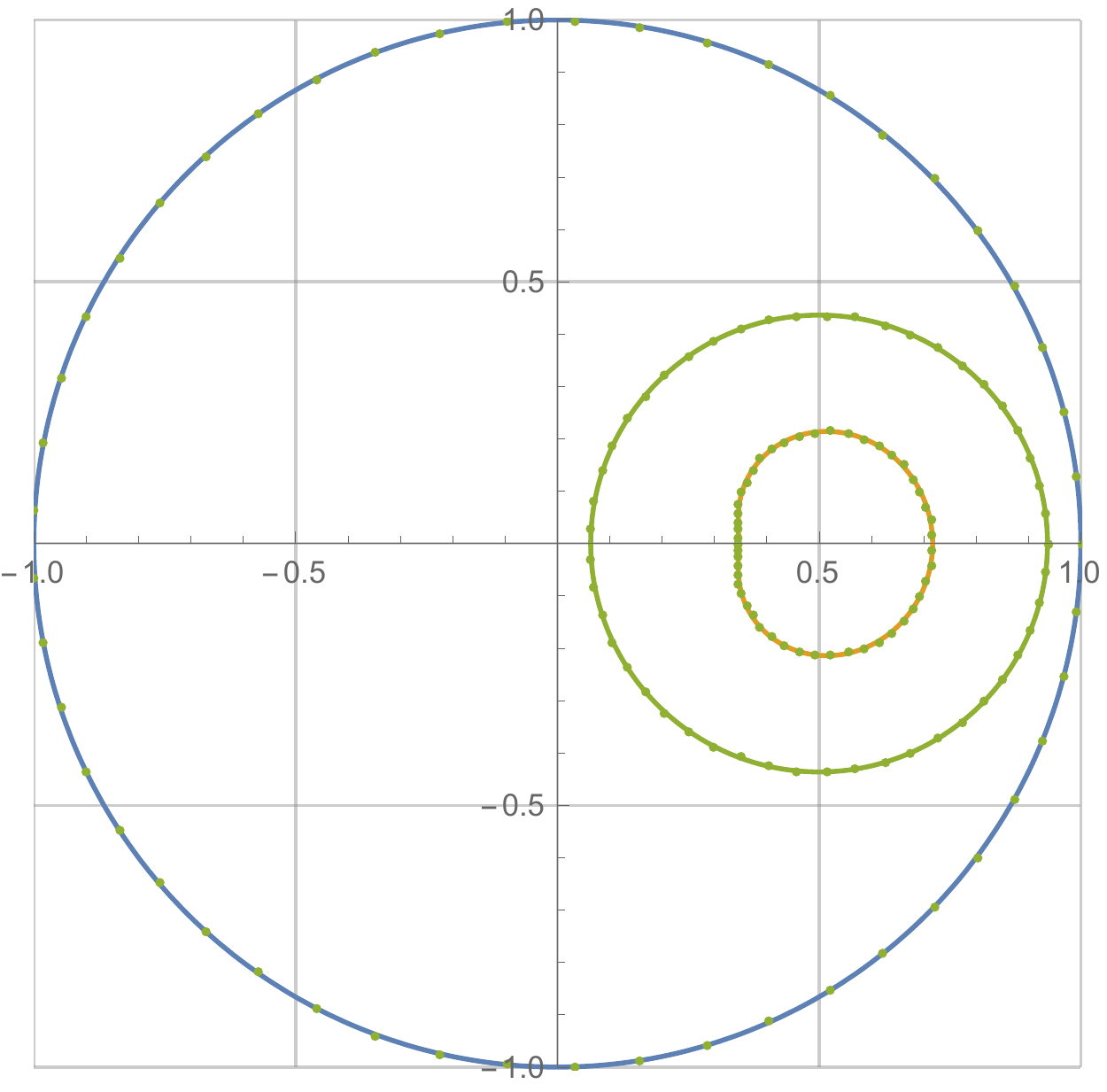}
\caption{Subordinate Schur functions}
\label{fig:subordinate-schur-functions}
\end{figure}
\end{center}
In Figure \ref{fig:subordinate-schur-functions}, the outermost circle
is the unit circle while the middle one is the image of
$|z|=0.9$ under $f(z)$ which is again a
circle with center at 1/2. The innermost figure is the image of
$|z|=0.9$ under $f^{(1/2)}(z;1)$.
\end{example}
\subsection{The change in Carath\'{e}odory function}
Let the Carath\'{e}odory function associated with the
perturbed Schur function $f^{(\beta_k)}(z;k)$ be
denoted by $\mathcal{C}^{(\beta_k)}(z;k)$.
Then, using
\eqref{eqn:perturbed Schur function},
we can write
\begin{align*}
\mathcal{C}^{(\beta_k)}(z;k)&=\dfrac{1+zf^{(\beta_k)}(z;k)}{1-zf^{(\beta_k)}(z;k)}\\
                            &=\dfrac{(\mathfrak{T}_{2,2}+z\mathfrak{T}_{1,2})+(\mathfrak{T}_{2,1}+z\mathfrak{T}_{1,1})f(z)}
                                    {(\mathfrak{T}_{2,2}-z\mathfrak{T}_{1,2})+(\mathfrak{T}_{2,1}-z\mathfrak{T}_{1,1})f(z)}.
\end{align*}
Further, using the relation
\eqref{eqn:relation between Schur and Carat functions},
we have
\begin{align}
\label{eqn: perturbed Carat function}
\mathcal{C}^{(\beta_k)}(z;k)=
\dfrac{\mathcal{Y}^{-}(z)+\mathcal{Y}^{+}(z)\mathcal{C}(z)}
      {\mathcal{W}^{-}(z)+\mathcal{W}^{+}(z)\mathcal{C}(z)},
\end{align}
where
\begin{align*}
\mathcal{Y}^{\pm}(z)&=
z(\mathfrak{T}_{(2,2)}+z\mathfrak{T}_{(1,2)})
\pm (\mathfrak{T}_{(2,1)}+z\mathfrak{T}_{(1,1)})\\
\mathcal{W}^{\pm}(z)&=
z(\mathfrak{T}_{(2,2)}-z\mathfrak{T}_{(1,2)})
\pm (\mathfrak{T}_{(2,1)}-z\mathfrak{T}_{(1,1)}).
\end{align*}
As an illustration, for the Schur function $f(z)=(1+z)/2$,
it is easy to verify that
\begin{align*}
\mathcal{C}(z)=\dfrac{2+z+z^2}{2-z-z^2}
\quad\mbox{and}\quad
\mathcal{C}^{(1/2)}(z;1)=\dfrac{2z^3-5z^2+7z+20}{-2z^3+7z^2-13z+20}.
\end{align*}
%
We plot these Carath\'{e}odory functions below.
\pagebreak
\begin{center}
\begin{figure}[h!]
  \begin{subfigure}[b]{0.4\textwidth}
    \includegraphics[width=\textwidth]{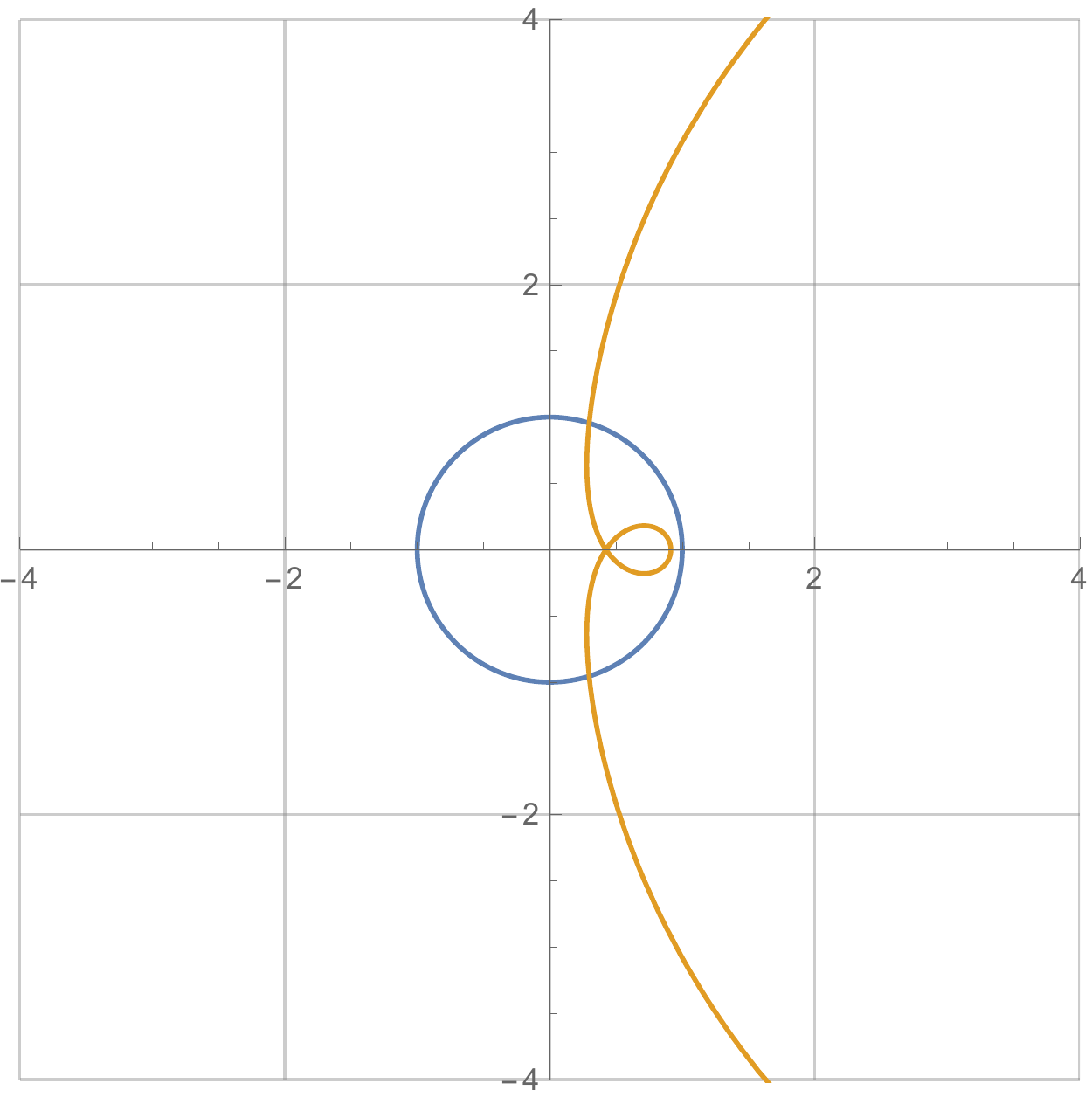}
    \caption{The function $\mathcal{C}(z)$.}
    \label{fig:original carat}
  \end{subfigure}
  \hfill
  \begin{subfigure}[b]{0.4\textwidth}
    \includegraphics[width=\textwidth]{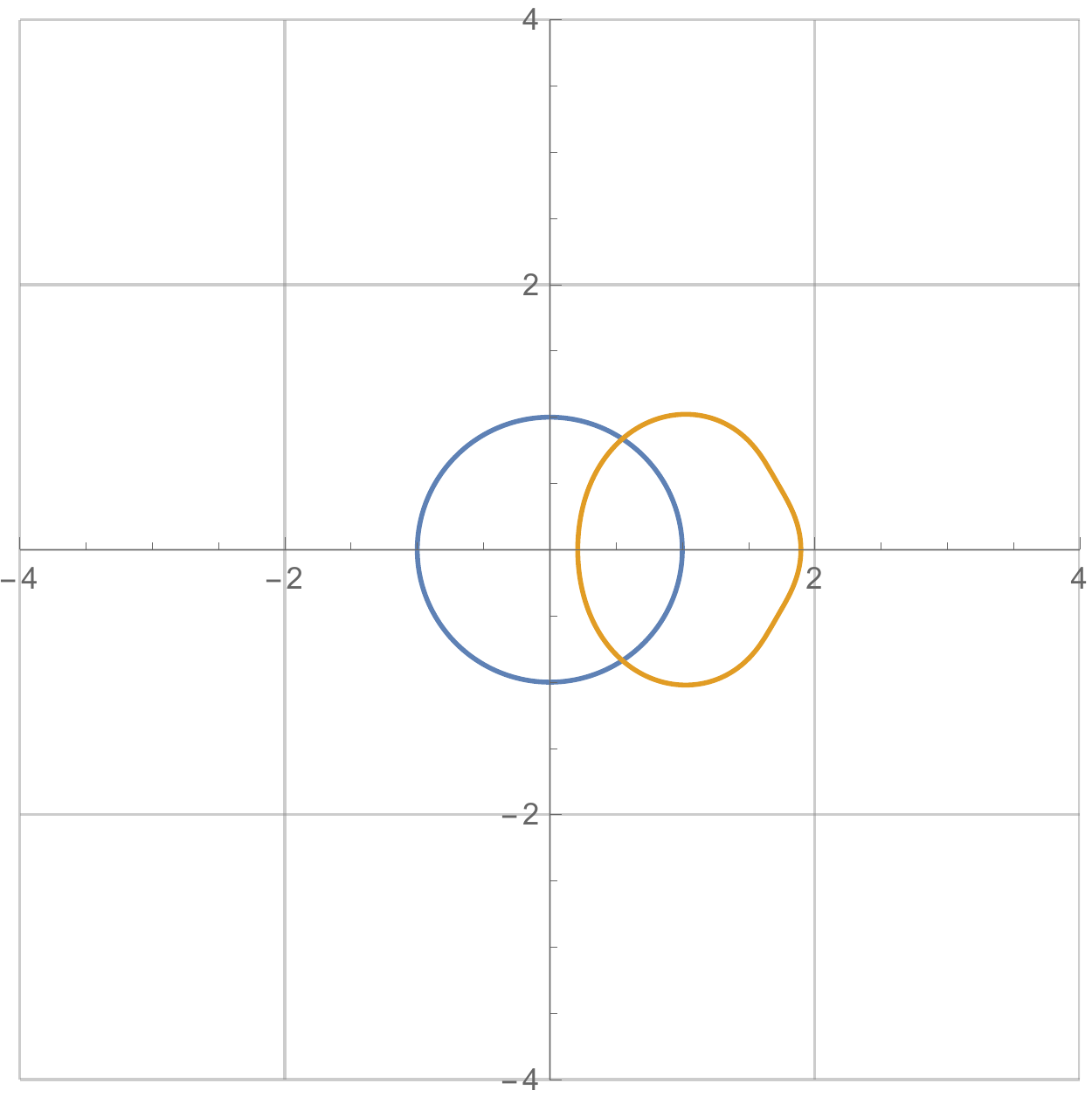}
    \caption{The function $\mathcal{C}^{(1/2)}(z;1)$.}
    \label{fig:perturbed carat}
  \end{subfigure}
  \caption{Perturbed mapping properties of Carath\'{e}odory functions.}
\end{figure}
\end{center}
In Figures  \ref{fig:original carat} and \ref{fig:perturbed carat},
the ranges of both the original and perturbed Carath\'{e}odory functions
are plotted for $|z|=0.9$.
Interestingly, the range of $\mathcal{C}(z)$ is unbounded
(Figure \ref{fig:original carat}) which is clear as $z=1$ is a pole of
$\mathcal{C}(z)$.
However $\mathcal{C}^{1/2}(z;1)$ has simple poles at 5/2 and
$(1\pm i\sqrt{15})/2$
and hence
with the use of perturbation we are able to make the
range bounded (Figure \ref{fig:perturbed carat}).
%
%


As shown in
\cite{Wall-cf-and-bdd-analytic-function-1944-BAMS},
the sequence
$\{\gamma_j\}_{j=0}^{\infty}$
satisfying the recurrence relation
\begin{align}
\label{eqn:recurrence relation for paramters in carat function}
\gamma_{p+1}=\dfrac{\gamma_p-\bar{\alpha}_p}{1-\alpha_p\gamma_p},
\quad p\geq0.
\end{align}
where $\gamma_0=1$ and $\alpha_j's$ are the Schur parameters
plays an important role in the $g$-fraction expansion for a special
class of Carath\'{e}odory functions.
Let $\{\gamma_j^{(\beta_k)}\}$ correspond to the perturbed
Carath\'{e}odory function $\mathcal{C}^{(\beta_k)}(z;k)$.
Since only $\alpha_j$ is perturbed,
it is clear that $\gamma_j$ remains unchanged for $j=0,1,\cdots,k$.
The first change,
$\gamma_{k+1}$ to $\gamma_{k+1}^{(\beta_k)}$,
occurs when $\alpha_k$ is replaced by $\beta_k$.
Consequently, $\gamma_{k+j}$, $j\geq2$,
change to
$\gamma_{k+j}^{(\beta_k)}$, $j\geq2$, respectively.
We now show that $\gamma_{j}$
can be expressed as a bilinear transformation
of $\gamma_j^{(\beta_k)}$ for $j\geq k+1$.
\begin{theorem}
Let $\{\gamma_j\}_{j=0}^{\infty}$ be the sequence corresponding to
$\{\alpha_j\}_{j=0}^{\infty}$ and $\{\gamma_j^{(\beta_k)}\}_{j=0}^{\infty}$
that to $\{\alpha_j^{(\beta_k)}\}_{j=0}^{\infty}$. Then,
\begin{align}
\label{eqn:blt for gamma and perturbed gamma used in theorem}
\gamma_{k+j}^{(\beta_k)}=\dfrac{\bar{a}_{k+j}\gamma_{k+j}-b_{k+j}}
                                  {-\bar{b}_{k+j}\gamma_{k+j}+a_{k+j}},
\quad j\geq1,
\end{align}
where
\begin{enumerate}[(i)]
\item $a_{k+1}=\dfrac{1-\bar{\alpha}_k\beta_k}{1-|\beta_k|^2}
       \quad\mbox{and}\quad
       b_{k+1}=\dfrac{\bar{\beta_k}-\bar{\alpha}_k}{1-|\beta_k|^2}$, (j=1).

\item For $j\geq2$,
      \begin{align*}
      \left(
                    \begin{array}{c}
                      a_{k+j} \\
                      b_{k+j} \\
                    \end{array}
                  \right)=
                  \dfrac{1}{1-|\alpha_{k+j-1}|^2}
                  \left(
                    \begin{array}{cc}
                      1 & \alpha_{k+j-1} \\
                     \bar{\alpha}_{k+j-1} & 1\\
                    \end{array}
                  \right)
                  \left(
                    \begin{array}{cc}
                      a_{k+j-1}-\bar{\alpha}_{k+j-1}\bar{b}_{k+j-1} \\
                      b_{k+j-1}-\bar{\alpha}_{k+j-1}\bar{a}_{k+j-1}  \\
                    \end{array}
                  \right).
      \end{align*}
\end{enumerate}
\end{theorem}
\begin{proof}
Consider first the expression
\begin{align*}
\dfrac{a_{k+1}\gamma_{k+1}^{(\beta_k)}+b_{k+1}}
      {\bar{b}_{k+1}\gamma_{k+1}^{(\beta_k)}+\bar{a}_{k+1}}.
\end{align*}
Substituting
$\gamma_{k+1}^{(\beta_k)}=(\gamma_k-\bar{\beta}_k)/(1-\beta_k\gamma_k)$
and the given values of $a_{k+1}$ and $b_{k+1}$,
it simplifies to
\begin{align*}
 &\dfrac{(1-\bar{\alpha}_{k}\beta_k)(\gamma_k-\bar{\beta}_k)+(\bar{\beta}_k-\bar{\alpha}_k)(1-\alpha_k\gamma_k)}
       {(\beta_k-\alpha_k)(\gamma_k-\bar{\beta}_k)+(1-\alpha_k\bar{\gamma}_k)(1-\beta_k\gamma_k)}
=\dfrac{\gamma_k(1-|\beta_k|^2)-\bar{\alpha}_k(1-|\beta_k|^2)}{(1-|\beta_k|^2)-\alpha_k\gamma_k(1-|\beta|^2)}
=\gamma_{k+1}.
\end{align*}
Since $|a_{k+1}|^2-|b_{k+1}|^2=(1-|\alpha_k|^2)/(1-|\beta_k|^2)\neq0$,
\eqref{eqn:blt for gamma and perturbed gamma used in theorem}
is proved for $j=1$.

Next, let
\begin{align*}
&\dfrac{a_{k+2}\gamma_{k+2}^{(\beta_k)}+b_{k+2}}
                {\bar{b}_{k+2}\gamma_{k+2}^{(\beta_k)}+\bar{a}_{k+2}}
=\dfrac{(a_{k+2}-\alpha_{k+1}b_{k+2})\gamma_{k+1}^{(\beta_k)}+(b_{k+2}-\bar{\alpha}_{k+1}a_{k+2})}
       {(\bar{b}_{k+2}-\alpha_{k+1}\bar{a}_{k+2})\gamma_{k+1}^{(\beta_k)}+(\bar{a}_{k+2}-\bar{\alpha}_{k+1}\bar{b}_{k+2})}
=\dfrac{N(\gamma_k)}{D(\gamma_k)}.
\end{align*}
Substituting first the given values of
$a_{k+2}$ and $b_{k+2}$,
the numerator becomes
\begin{align*}
N(\gamma_k)=(1-|\alpha_{k+1}|^2)[(a_{k+1}-\bar{\alpha}_{k+1}\bar{b}_{k+1})\gamma_{k+1}^{(\beta_k)}+
                                 (b_{k+1}-\bar{\alpha}_{k+1}\bar{a}_{k+1})],
\end{align*}
and then using,
$\gamma_{k+1}^{(\beta_k)}=(\bar{a}_{k+1}-b_{k+1})/
(-\bar{b}_{k+1}\gamma_{k+1}+a_{k+1})$,
\begin{align*}
N(\gamma_k)=(1-|\alpha_{k+1}|^2)(|\alpha_{k+1}|^2-
|b_{k+1}|^2)(\gamma_{k+1}-\bar{\alpha}_{k+1}).
\end{align*}
With similar calculations, we obtain
\begin{align*}
D(\gamma_k)=(1-|\alpha_{k+1}|^2)(|\alpha_{k+1}|^2-
|b_{k+1}|^2)(1-\alpha_{k+1}\gamma_{k+1}).
\end{align*}
This means
\begin{align*}
\dfrac{N(\gamma_k)}{D(\gamma_k)}=
\dfrac{\gamma_{k+1}-\bar{\alpha}_{k+1}}
{1-\alpha_{k+1}\gamma_{k+1}}=
\gamma_{k+2},
\end{align*}
where
$|a_{k+2}|^2-|b_{k+2}|^2=|a_{k+1}|^2-|b_{k+1}|^2\neq0$,
thus proving
\eqref{eqn:blt for gamma and perturbed gamma used in theorem}
for $j=2$.
The remaining part of the proof
is follows by a simple induction on $j$.
\end{proof}
\begin{remark}
With the condition $|\alpha_p|<1$, it is clear that
\eqref{eqn:recurrence relation for paramters in carat function}
gives analytic self-maps of the unit disk.
Similar to the changes in mapping properties obtained as
a consequence of perturbation,
it is expected that
\eqref{eqn:blt for gamma and perturbed gamma used in theorem}
may lead to interesting results in fractal geometry and complex dynamics.
\end{remark}
\begin{remark}
Since $\gamma_{j}$ and $\gamma_j^{(\beta_k)}$ are related by a bilinear
transformation, it is clear that the expressions for $a_{k+j}$ and
$b_{k+j}$, $j\geq1$, are not unique.
\end{remark}
It is known that if $\mathcal{C}(z)$ is real for real $z$, then
the $\alpha_p's$ are real and
$\gamma_p=1$, $p=0,1,\cdots$.
Further, it is clear from
\eqref{eqn:blt for gamma and perturbed gamma used in theorem}
that $\gamma_j^{(\beta_k)}=1$ whenever $\gamma_j=1$.
In this case, the following $g$-fraction is obtained
for $k\geq0$.
\begin{align*}
\dfrac{1-z}{1+z}\mathcal{C}^{(\beta_k)}(z)=
\dfrac{1}{1}
\begin{array}{cc}\\$-$\end{array}
\dfrac{g_1\omega}{1}
\begin{array}{cc}\\$-$\end{array}
\dfrac{(1-g_1)g_2\omega}{1}
\begin{array}{cc}\\$-$\end{array}
\cdots
\dfrac{(1-g_k)g_{k+1}^{(\beta_k)}\omega}{1}
\begin{array}{cc}\\$-$\end{array}
\dfrac{(1-g_{k+1}^{(\beta_k)})g_{k+2}\omega}{1}
\begin{array}{cc}\\$-$\end{array}
\cdots
\end{align*}
where $g_j=(1-\alpha_{j-1})/2$, $j=1,\cdots,k,k+2\cdots$,
$g_{k+1}^{(\beta_k)}=(1-\beta_k)/2$
and $\omega=-4z/(1-z)^2$.
\section{A class of Pick functions and Schur functions}
\label{sec:a class of pick functions and schur functions}
Let the Hausdorff sequence $\{\nu_j\}_{j\geq0}$
with $\nu_0=1$ be given so that there
exists a bounded non-decreasing measure $\nu$ on [0,1]
satisfying
\begin{align*}
\nu_j=\int_{0}^{1}\sigma^jd\nu(\sigma),
\quad j\geq0.
\end{align*}
By \cite[Theorem 69.2]{Wall_book}, the existence of
$d\nu(\sigma)$ is equivalent to the power series
\begin{align*}
F(\omega)=\sum_{j\geq0}\nu_j{\omega}^j=
\int_{0}^{1}\dfrac{1}{1-\sigma \omega}d\nu(\sigma)
\end{align*}
having a continued fraction
expansion of the form
\begin{align*}
\int_{0}^{1}\dfrac{1}{1-\sigma \omega}d\nu(\sigma)=
\dfrac{\nu_0}{1}
\begin{array}{cc}\\$-$\end{array}
\dfrac{(1-g_0)g_1\omega}{1}
\begin{array}{cc}\\$-$\end{array}
\dfrac{(1-g_1)g_2\omega}{1}
\begin{array}{cc}\\$-$\end{array}
\dfrac{(1-g_2)g_3\omega}{1}
\begin{array}{cc}\\$-$\end{array}
\cdots
\end{align*}
where $\nu_0\geq0$ and $0\leq g_p\leq 1$, $p\geq0$.
Such functions $F(\omega)$ are analytic in the slit domain
$\mathbb{C}\setminus[1,\infty)$
and belong to the class of Pick functions.
We note that the Pick functions are analytic in the
upper half plane and have a positive imaginary part
\cite{Donoghue-interpolation-pick-function-rocky-mountain}.

In the next result, we characterize some members of the class of Pick
functions using the gap $g$-fraction $\mathcal{F}_2^{(a,b,c)}(\omega)$.
The proof is similar to that of
\cite[Theorem 1.5]{Kustner-g-fractions-2002-CMFT} and
follows from \cite[Lemma 3.1]{Kustner-g-fractions-2002-CMFT}, given earlier
as \cite[Corollary 2.1]{Merkes-AMS-1959-typically-real}.
\begin{theorem}
\label{thm:characterising members of pick functions}
If $a, b, c\in\mathbb{R}$ with $-1<a\leq c$ and $0\leq b\leq c$,
then the functions
\begin{align*}
\omega &\mapsto \dfrac{F(a+1,b+1;c+1;\omega)}{F(a+1,b;c+1;\omega)}
\quad ; \quad
\omega \mapsto \dfrac{\omega F(a+1,b+1;c+1;\omega)}{F(a+1,b;c+1;\omega)}\\
\omega &\mapsto \dfrac{F(a+2,b+1;c+2;\omega)}{F(a+1,b;c+1;\omega)}
\quad ; \quad
\omega \mapsto \dfrac{zF(a+2,b+1;c+2;\omega)}{F(a+1,b;c+1;\omega)}\\
\omega &\mapsto \dfrac{F(a+2,b+1;c+2;\omega)}{F(a+1,b+1;c+1;\omega)}
\quad ; \quad
\omega \mapsto \dfrac{\omega F(a+2,b+1;c+2;\omega)}{F(a+1,b+1;c+1;\omega)}
\end{align*}
are analytic in $\mathbb{C}\setminus[1,\infty)$ and each
function map both the open unit disk $\mathbb{D}$ and the
half plane $\{\omega\in\mathbb{C}: \mathrm{Re}\,\omega<1\}$ univalently
onto domains that are convex in the direction of the
imaginary axis.
%
%
\end{theorem}
We would like to note here that by a domain convex in the direction
of imaginary axis, we mean that every line parallel to the imaginary
axis has either connected or empty intersection with the
corresponding domain
\cite{Duren-book},
(see also
\cite{Swami-mapping-prop-BHF-2014-JCA, Kustner-g-fractions-2002-CMFT}.
).
\begin{proof}
[Proof of Theorem \ref{thm:characterising members of pick functions}]
With the given restrictions on $a$, $b$ and $c$, $\mathcal{F}_{2}^{(a,b,c)}(\omega)$
has a $g$-fraction expansion and hence by \cite[Theorem 69.2]{Wall_book}, there
exists a non-decreasing function $\nu_0:[0,1]\mapsto [0,1]$ with a total increase of 1 and
\begin{align*}
\dfrac{F(a+1,b+1;c+1;\omega)}{F(a+1,b;c+1;\omega)}=
\int_{0}^{1}\dfrac{1}{1-\sigma \omega}d\nu_0(\sigma),
\quad
\omega\in\mathbb{C}\setminus[1,\infty),
\end{align*}
which implies
\begin{align*}
\dfrac{\omega F(a+1,b+1;c+1;\omega)}{F(a+1,b;c+1;\omega)}=
\int_{0}^{1}\dfrac{\omega}{1-\sigma \omega}d\nu_0(\sigma),
\quad
\omega\in\mathbb{C}\setminus[1,\infty).
\end{align*}
Now, if we define
\begin{align*}
\nu_1(\sigma)=\dfrac{1}{k_2}\int_{0}^{\sigma}\rho d\nu_0(\rho),
\end{align*}
where $k_2=(a+1)/(c+1)>0$, it can be easily seen that
$\nu_1:[0,1]\mapsto [0,1]$ is again a non-decreasing
map with $\nu_1(1)-\nu_1(0)=1$.
Further, using the contiguous relation
\begin{align}
\label{eqn:contiguos relation in schur function and gap g fraction relation}
F(a+1,b;c;\omega)-F(a,b;c;\omega)=\dfrac{b}{c}\omega F(a+1,b+1;c+1;\omega),
\end{align}
we obtain
\begin{align*}
\dfrac{\omega F(a+2,b+1;c+2;\omega)}{F(a+1,b;c+1;\omega)}=
\int_{0}^{1}\dfrac{\omega}{1-\sigma \omega}d\nu_1(\sigma),
\quad \omega\in\mathbb{C}\setminus[1,\infty),
\end{align*}
and hence
\begin{align*}
\dfrac{F(a+1,b+1;c+1;\omega)}{F(a+1,b;c+1;\omega)}=
1+k_2
\int_{0}^{1}\dfrac{\omega}{1-\sigma \omega}d\nu_1(\sigma),
\quad \omega\in\mathbb{C}\setminus[1,\infty).
\end{align*}
Further, noting that the coefficient of $\omega$ in
$F(a+2,b+1;c+2;\omega)/F(a+1,b;c+1;\omega)$ is
$[(b+1)(c-a)]/[(c+1)(c+2)]=k_3+(1-k_3)k_2$,
we define
\begin{align*}
\nu_2(\sigma)=\dfrac{1}{k_3+k_2(1-k_3)}
\int_{0}^{\sigma}\rho d\nu_1(\rho),
\end{align*}
and find that
\begin{align*}
\dfrac{F(a+2,b+1;c+2;\omega)}{F(a+1,b;c+1;\omega)}=
1+[k_3+k_2(1-k_3)]
\int_{0}^{1}\dfrac{\omega}{1-\sigma \omega}d\nu_2(\sigma).
\end{align*}
Finally from Gauss continued fraction
\eqref{eqn:Gauss continued fraction},
we conclude that
$F(a+2,b+1;c+2;\omega)/F(a+1,b+1;c+1;\omega)$ has a $g$-fraction
expansion and so there exists a map $\nu_3:[0,1]\mapsto[0,1]$
which is non-decreasing, $\nu_3(1)-\nu_3(0)=1$ and
\begin{align*}
\dfrac{\omega F(a+2,b+1;c+2;\omega)}{F(a+1;b+1;c+1;\omega)}=
\int_{0}^{1}\dfrac{\omega}{1-\sigma \omega}d\nu_{3}(\sigma),
\quad \omega\in\mathbb{C}\setminus [1,\infty).
\end{align*}
Defining for $a<c$
\begin{align*}
\nu_4(\sigma)=\dfrac{1}{(1-k_2)k_3}
\int_{0}^{\sigma}\rho d\nu_3(\rho),
\end{align*}
so that $(1-k_2)k_3>0$, and using the fact that the coefficient
of $\omega$ in $F(a+2,b+1;c+2;\omega)/F(a+1,b+1;c+1;\omega)$ is $(1-k_2)k_3$,
we obtain
\begin{align*}
\dfrac{F(a+2,b+1;c+2;\omega)}{F(a+1,b+1;c+1;\omega)}=
1+[(1-k_2)k_3]
\int_{0}^{1}\dfrac{\omega}{1-\sigma \omega}d\nu_4(\sigma).
\end{align*}
Thus, with $\nu_j$, $j=0,1,2,3,4$, satisfying
the conditions of
\cite[Lemma 3.1]{Kustner-g-fractions-2002-CMFT},
\cite[Corollary 2.1]{Merkes-AMS-1959-typically-real},
the proof of the theorem is completed.
%
%
%
\end{proof}
\begin{remark}
Ratios of Gaussian hypergeometric functions having mapping properties described in Theorem $\ref{thm:characterising members of pick functions}$
are also found in \cite[Theorem 1.5]{Kustner-g-fractions-2002-CMFT}
but for the ranges $-1\leq a\leq c$ and $0<b\leq c$. Hence for the common range $-1<a\leq c$ and $0<b\leq c$, two different ratios of hypergeometric functions belonging to the class of Pick functions can be obtained leading to the expectation of finding more such ratios for every possible range.

\end{remark}

It may be noted that the ratio of Gaussian hypergeometric functions in \eqref{eqn:Kustner continued fraction}
denoted here as $\mathcal{F}(z)$
has the mapping properties given in Theorem
$\ref{thm:characterising members of pick functions}$,
which is proved in \cite[Theorem 1.5]{Kustner-g-fractions-2002-CMFT}.
We now consider its $g$-fraction expansion with the parameter $k_2$ missing.
Using the contiguous relation
\eqref{eqn:contiguos relation in schur function and gap g fraction relation}
and the notations used in Theorems
$\ref{thm:structural relation for gk missing}$
and
$\ref{thm:gap g fraction identity with minimal parameters}$,
it is clear that
$\mathcal{F}_3^{(a,b,c)}(\omega)=F(a+2,b+1;c+2;\omega)/F(a+1,b+1;c+2;\omega)$
and
\begin{align*}
\mathcal{H}_3(\omega)=
1-\dfrac{1}{\mathcal{F}_3^{(a,b,c)}(w)}=
\dfrac{b+1}{c+2}\omega\dfrac{F(a+2,b+2;c+3;\omega)}
{F(a+2,b+1;c+2;\omega)}.
\end{align*}
Then
\begin{align*}
h(2;\omega)=(1-k_1)\mathcal{H}_3(\omega)=
\dfrac{(c-b)(b+1)}{(c)(c+2)}\omega\dfrac{F(a+2,b+2;c+3;\omega)}
{F(a+2,b+1;c+2;\omega)}
\end{align*}
Then, from Theorem $\ref{thm:structural relation for gk missing}$,
\begin{align*}
\mathcal{F}(2;\omega)
&=\dfrac{1}{1-(1-k_0)k_1\omega}-
\dfrac{(1-k_0)k_1\omega h(2;\omega)}{[1-(1-k_0)k_1z]h(2;\omega)-[1-(1-k_0)k_1z]^2}\\
&=\dfrac{c}{c-bz}-\dfrac{bczh(2;z)}{c(c-bz)h(2;z)-(c-bz)^2}
\end{align*}
which implies
\begin{align*}
\mathcal{F}(2;\omega)=
\dfrac{c}{c-b\omega}-
\dfrac{\dfrac{b(b+1)(c-b)}{c+2}\omega^2\dfrac{F(a+2,b+2;c+3;\omega)}
{F(a+2,b+1;c+2;\omega)}}
{\dfrac{(c-b)(b+1)(c-b\omega)}{c+2}\omega
\dfrac{F(a+2,b+2;c+3;\omega)}{F(a+2,b+1;c+2;\omega)}-(c-b\omega)^2}
\end{align*}
that is $\mathcal{F}(2;\omega)$ is given as a rational transformation of a new ratio of hypergeometric functions. It may also be noted that
for $-1\leq a\leq c$ and $0<b\leq c$,
\cite[Theorem 1.5]{Kustner-g-fractions-2002-CMFT},
both $\mathcal{F}(\omega)$ and
$\mathcal{F}(2;\omega)$
will map both the unit disk $\mathbb{D}$ and the half plane
$\{\omega\in\mathbb{C}:\mathrm{Re}\,\omega<1\}$ univalently onto domains that are convex in the direction of the imaginary axis.

As an illustration, we plot both these functions in figures
\eqref{fig:kustner ratio image} and \eqref{fig:own ratio image}.
\begin{center}
\begin{figure}[h!]
  \begin{subfigure}[b]{0.38\textwidth}
    \includegraphics[width=\textwidth]{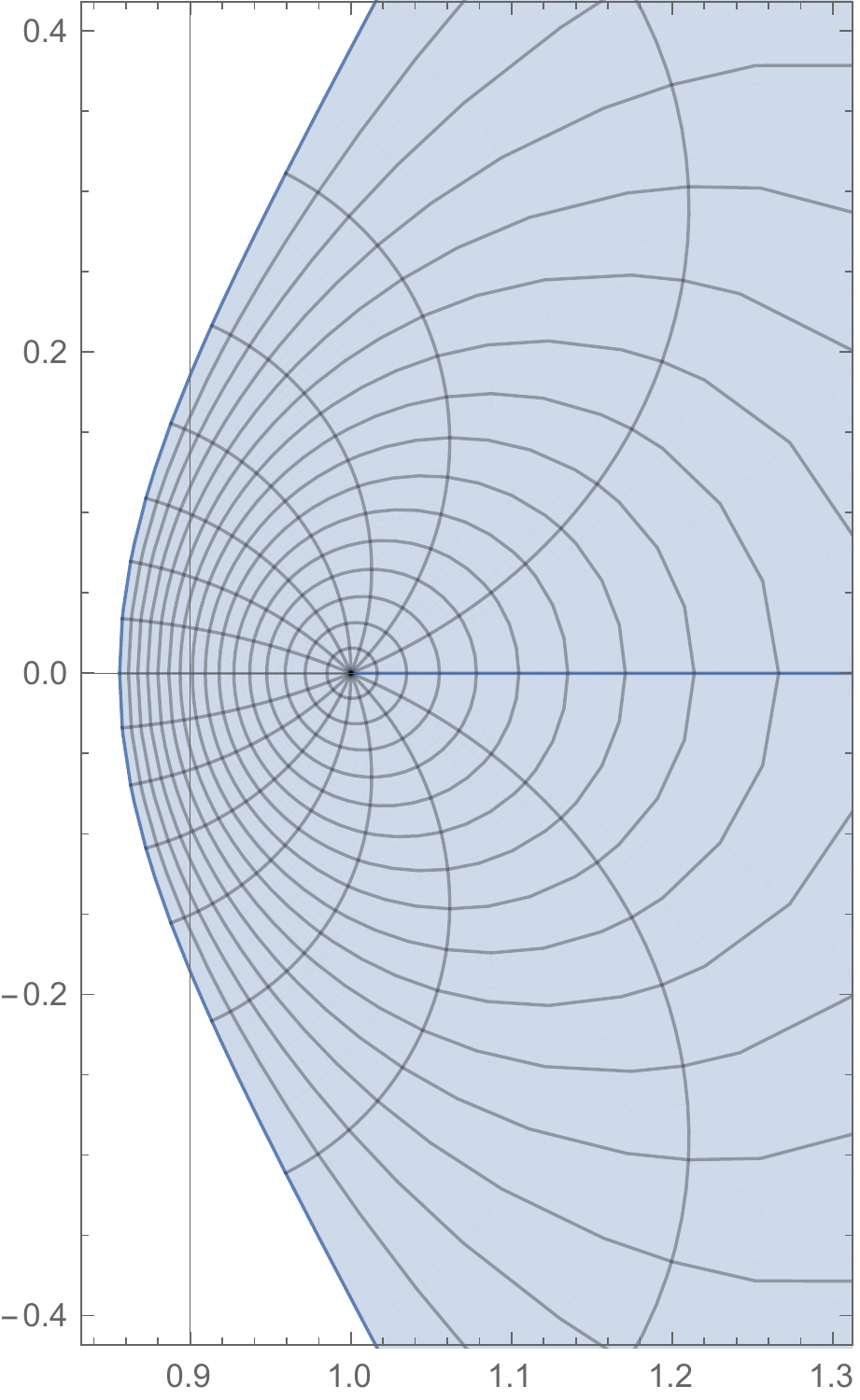}
    \caption{The function $\mathcal{F}(\omega)$}
    \label{fig:kustner ratio image}
  \end{subfigure}
  \hfill
  \begin{subfigure}[b]{0.4\textwidth}
    \includegraphics[width=\textwidth]{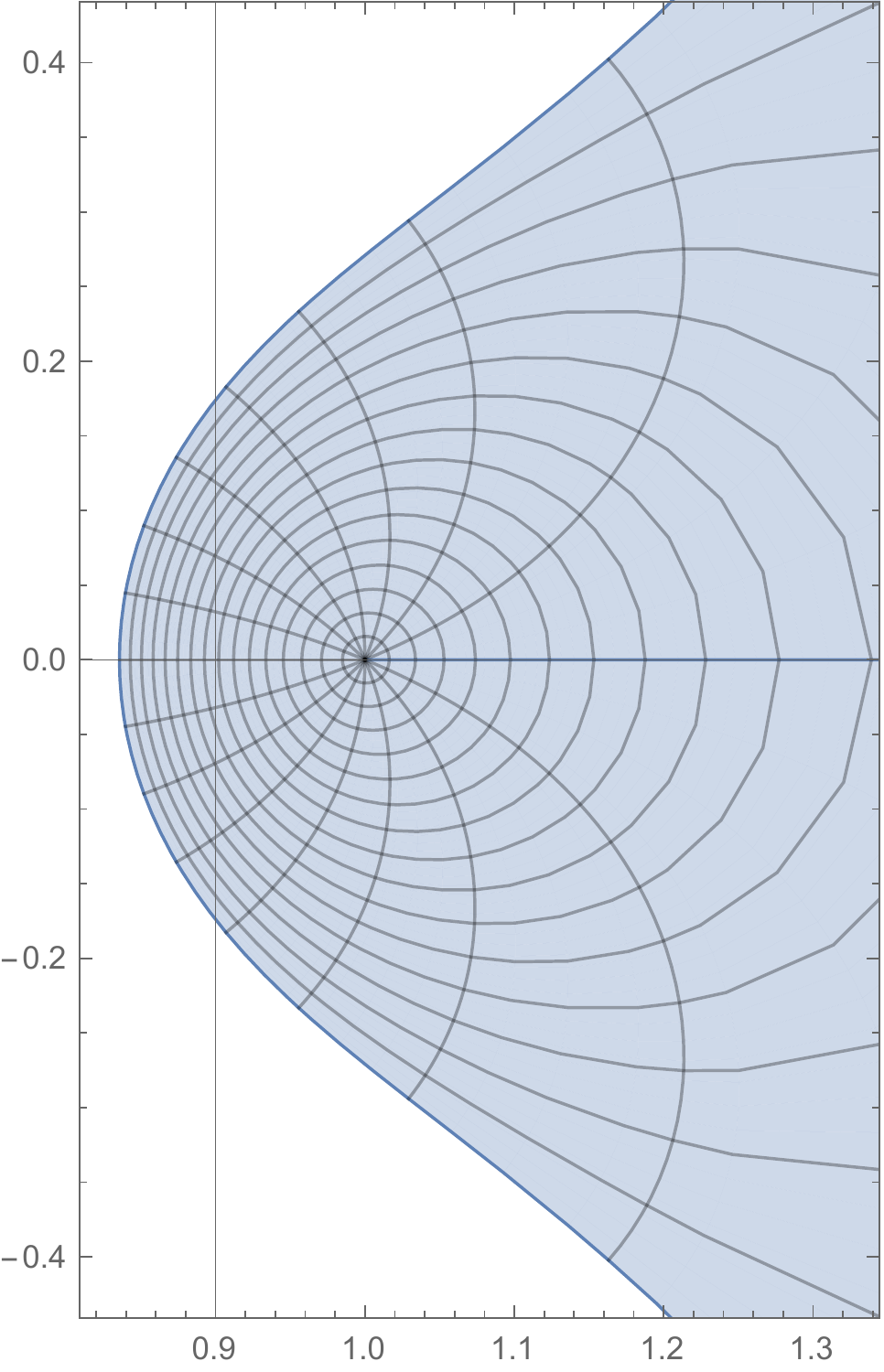}
    \caption{The function $\mathcal{F}(2;\omega)$}
    \label{fig:own ratio image}
  \end{subfigure}
  \caption{The images of the disc $|\omega|<0.999$ under the mappings $\mathcal{F}(\omega)$ and  $\mathcal{F}(2;z)$ for $a=0$, $b=0.1$,
  $c=0.4$.}
\end{figure}
\end{center}
\subsection{A class of Schur functions}
From Theorem $\ref{thm:gap g fraction identity with minimal parameters}$
we obtain
\begin{align*}
\dfrac{k_{2j+1}\omega}{1-\dfrac{(1-k_{2j+1})k_{2j+2}\omega}
                      {1-\dfrac{(1-k_{2j+2})k_{2j+3}\omega}
                      {1-\cdots}}}
&=1-\dfrac{F(a+j,b+j;c+2j;\omega)}{F(a+j+1,b+j;c+2j;\omega)}\\
&=\dfrac{b+j}{c+2j}\dfrac{\omega F(a+j+1,b+j+1;c+2j+1;\omega)}
                         {F(a+j+1,b+j,c+2j;\omega)}
\end{align*}
where the last equality follows from the contiguous relation
\eqref{eqn:contiguos relation in schur function and gap g fraction relation}
Hence using \cite[eqns. 3.3 and 5.1]{Wall-cf-and-bdd-analytic-function-1944-BAMS}
we get
\begin{align*}
\dfrac{1-z}{2}\dfrac{1-f_{2j}(z)}{1+zf_{2j}(z)}=
\dfrac{b+j}{c+2j}\dfrac{F(a+j+1,b+j+1;c+2j+1;\omega)}
{F(a+j+1,b+j,c+2j;\omega)},
\quad j\geq1,
\end{align*}
where $f_n(z)$ is the Schur function and
$\omega$ and $z$ are related as $\omega=-4z/(1-z)^2$.
Similarly, interchanging $a$ and $b$ in
\eqref{eqn:contiguos relation in schur function and gap g fraction relation}
we obtain
\begin{align*}
\dfrac{1-z}{2}\dfrac{1-f_{2j+1}(z)}{1+zf_{2j+1}(z)}=
\dfrac{a+j+1}{c+2j+1}\dfrac{F(a+j+2,b+j+1;c+2j+2;\omega)}
{F(a+j+1,b+j+1,c+2j+1;\omega)},
\quad j\geq0,
\end{align*}
where $\omega=-4z/(1-z)^2$.

Moreover, using the relation $\alpha_{j-1}=1-2k_j$, $j\geq1$,
the related sequence of Schur parameters is given by
\begin{align*}
\alpha_j=
\left\{
  \begin{array}{ll}
    \dfrac{c-2b}{c+j}, & \hbox{$j=2n$, $n\geq0$;} \\
    \dfrac{c-2a-1}{c+j}, & \hbox{$j=2n+1$, $n\geq1$.}
  \end{array}
\right.
\end{align*}
We note the following particular case. For $a=b-1/2$ and
$c=b$, the resulting Schur parameters are
$\alpha_{j}^{(b)}=-b/(b+j)$, $j\geq0$.
Such parameters have been considered in
\cite{Ranga-szego-polynomials-2010-AMS}
(when $b\in\mathbb{R}$) in the
context of orthogonal polynomials on the unit circle.
These polynomials are known in modern literature as
Szeg\"{o} polynomials and we suggest the interested readers to
refer \cite{Szego-book} for further information.

Finally, as an illustration we note that while the Schur function
associated with the parameters
$\{\alpha_j^{(b)}\}_{j\geq0}$ is
$f(z)=-1$,
that associated with the parameters
$\{\alpha^{(b)}_j\}_{j\geq1}$
is given by
\begin{align*}
\dfrac{1-z}{2}\dfrac{1-f^{(b)}(z)}{1+zf^{(b)}(z)}=
\dfrac{b+1/2}{b+1}\dfrac{F(b+3/2,b+1;b+2;\omega)}
                        {F(b+1/2,-;-;\omega)}
\end{align*}
where $\omega=-4z/(1-z)^2$.

We remark that in this section, specific illustration of the results given in Section~\ref{sec:gap g fractions and structural relations} are discussed leading to characterizing a class of ratio of hypergeometric functions. Similar characterization of functions involving the function $\omega=-4z/(1-z)^2$, given in Section~\ref{sec:Perturbed Schur parameters}
may provide some important consequences of such perturbation of $g$-fractions.

\end{document}